\documentclass[11pt,reqno]{amsproc}

\usepackage[utf8]{inputenc}
\usepackage[margin=1in]{geometry}
\usepackage{amsmath}
\usepackage{amssymb}
\usepackage{amsfonts}
\usepackage{amsthm}
\usepackage{mathrsfs}
\usepackage{mathtools}
\usepackage{enumitem}
\usepackage[dvipsnames]{xcolor}

\usepackage{hyperref}
\hypersetup{
    colorlinks=true,
    citecolor=magenta,
    linkcolor=blue,
    filecolor=teal,
    urlcolor=cyan,
    pdftitle={Regularity for axisymmetric Navier-Stokes with an Euler length},
    pdfauthor={Peter Constantin, Mihaela Ignatova, and Vlad Vicol},
    pdfsubject={Axisymmetric Navier-Stokes regularity},
    pdfkeywords={Navier-Stokes equations, local regularity, axial symmetry, ancient vorticity limits},
    }
\usepackage{graphicx}
\usepackage{tikz}

\numberwithin{equation}{section}

 \newtheorem{theorem}{Theorem}[section]
 \newtheorem{proposition}[theorem]{Proposition}
 \newtheorem{lemma}[theorem]{Lemma}
 \newtheorem{corollary}[theorem]{Corollary}
 \newtheorem{definition}[theorem]{Definition}
 \theoremstyle{remark}
 \newtheorem{remark}[theorem]{Remark}
\def\p{\partial}

\newcommand{\norm}[1]{\left\|#1\right\|}

\newcommand*{\curl}{\ensuremath{\mathrm{curl\,}}}
\newcommand*{\supp}{\ensuremath{\mathrm{supp\,}}}
\renewcommand*{\div}{\ensuremath{\mathrm{div\,}}}

\newcommand*{\RR}{\ensuremath{\mathbb{R}}}
\newcommand{\eps}{\varepsilon}

\newcommand{\Ecal}{\mathcal{E}}

\newcommand{\Dcal}{\mathcal{D}}

\newcommand{\Ocal}{\mathcal{O}}

\newcommand{\Ucal}{\mathcal{U}}

\newcommand{\esssup}{\operatorname*{ess\,sup}}

\title[Axisymmetric Navier-Stokes with an Euler length]{Regularity for axisymmetric Navier-Stokes with an Euler length}

\author{Peter Constantin} 
\address{Department of Mathematics, Princeton University, Princeton, NJ 08540}
\email{const@math.princeton.edu}
\urladdr{https://web.math.princeton.edu/~const/}

\author{Mihaela Ignatova}
\address{Department of Mathematics, Temple University, Philadelphia, PA 19122}
\email{ignatova@temple.edu}
\urladdr{https://sites.temple.edu/ignatova/}

\author{Vlad Vicol}
\address{Department of Mathematics, Courant Institute, New York University, New York, NY, 10012}
\email{vicol@cims.nyu.edu}
\urladdr{https://cims.nyu.edu/~vicol/}

\date{}

\begin{document}
\begin{abstract}
We prove local regularity for axisymmetric suitable weak solutions of the 3D Navier-Stokes equations, which are smooth before the terminal time $t=0$, and satisfy Type~II pointwise bounds at a vanishing length scale $\ell(t)$. We say $\ell(t)$ is an \emph{Euler length} if it is non-increasing, satisfies a doubling condition, and if $\ell(t)\to0$ and $(-t)/\ell(t)^2\to0$ as $t\to0^-$. This includes power laws $\ell(t)=(-t)^\gamma$ with $0<\gamma<1/2$, and logarithmic enlargements of the parabolic length. Our main result shows that local bounds of the type  $|u(\cdot,t)| \leq C\ell(t)/(-t)$ and $|\nabla^2 u(\cdot,t)|\leq C/((-t)\ell(t))$ for all $t\in (-1,0)$ imply regularity. The proof adapts the circulation and potential-vorticity argument of our earlier paper~\cite{CIV26} to ancient limits obtained from the rescaled vorticity system by zooming in. We show that the second derivative a priori assumption may be replaced by a H\"older bound on the azimuthal vorticity and a corresponding bound on its potential vorticity.
%\hfill \today
\end{abstract}
\maketitle
%\setcounter{tocdepth}{1}
%\tableofcontents

\section{Introduction}
\label{sec:intro}

Whether an axisymmetric finite energy solution of the unforced 3D Navier-Stokes equations, arising from smooth and decaying initial data, can develop a finite time singularity remains an open problem. 

In this paper, we consider the 3D Navier-Stokes equations without external forcing
\begin{subequations}
\label{eq:nse}
\begin{align}
\p_t u+(u\cdot\nabla)u-\Delta u+\nabla\pi=0\,,
\\ 
\div u=0\,,
\end{align}
\end{subequations}
where $u\colon Q\to\RR^3$ is the velocity, $\pi\colon Q\to\RR$ is the pressure, and $Q=B(1)\times(-1,0)$ is an open unit parabolic cylinder.\footnote{Here and throughout this paper, we use the notation $B(\rho)=\{x\in\RR^3:|x|<\rho\}$ and $Q(\rho)=B(\rho)\times(-\rho^2,0)$.} 
Throughout the paper we take $(u,\pi)$ to be a suitable weak solution of~\eqref{eq:nse} in the sense of Caffarelli, Kohn, and Nirenberg~\cite{CKN}, with $\pi\in L^{3/2}(Q)$, which is smooth on compact subsets of $Q$. We assume no regularity at time $t=0$. Moreover, we assume throughout the paper that the solution $(u,\pi)$ of~\eqref{eq:nse} is axisymmetric about the $z$-axis. With $x=(x',z)$, $r=|x'|$, and the usual cylindrical frame $(e_r,e_\theta,e_z)$, we set
\begin{equation*}
u = u_r e_r + u_\theta e_\theta + u_z e_z = b + u_\theta e_\theta\,,
\end{equation*}
where $u_r, u_\theta, u_z$ depend only on $(r,z,t)$. We denote by $b$ the meridional velocity.

\subsection{Main result}
In order to state our main result, we introduce:
\begin{definition}[{\bf Euler length}]
\label{def:Euler:length}
A continuous function $\ell \colon (-1,0) \to (0,\infty)$ is called an \emph{Euler length} if it is non-increasing, it satisfies the doubling property 
\begin{subequations}
\label{eq:pointwise:length}
\begin{equation}
\label{eq:pointwise:length:doubling}
 \ell(2t)\leq D\ell(t)\,, \qquad \mbox{for all }t\in (-1/2,0)\,,
\end{equation}
for a suitable constant $D \in [1,\infty)$, it asymptotically vanishes
\begin{equation}
\lim_{t\uparrow0}\ell(t)=0\,,
\end{equation}
and it satisfies
\begin{equation}
\label{visczero}
\lim_{t\uparrow0}\frac{(-t)}{\ell(t)^2}=0.
\end{equation}
\end{subequations}
\end{definition}

The term ``Euler length'' is motivated in part by the requirement  \eqref{visczero} that the viscosity coefficient under Euler rescaling vanishes asymptotically, so formally the Euler behavior emerges in the limit.

Examples of Euler lengths are $\ell(t)=(-t)^\gamma$ for $\gamma \in (0,1/2)$ and $\ell(t)= (-t)^{1/2} (\log(e/(-t)))^a$ with $0<a\leq 1/2$. 

Recall that the point $(0,0)$ is called \emph{regular} if
\begin{align}
\esssup_{Q(\rho)}|u|<\infty
\quad\mbox{for some }\rho>0
\,.
\label{eq:intro:regular}
\end{align}
More generally, for $x_0\in B(1)$ the point $(x_0,0)$ is called \emph{regular} if~\eqref{eq:intro:regular} holds with $Q(\rho)$ replaced by $B(x_0,\rho)\times(-\rho^2,0)\subset Q$, and \emph{singular} otherwise.
The simplest way to state the main result of this paper is:

\begin{theorem}[{\bf Local regularity under Type~II bounds at an Euler length}]
\label{thm:pointwise:variable}
Let $(u,\pi)$ be a suitable weak solution of~\eqref{eq:nse} in $Q$, which is smooth on compact subsets of $Q$, and is axisymmetric. Let $\ell = \ell(t)$ be an Euler length, in the sense of Definition~\ref{def:Euler:length}. Suppose that there exist constants $C_0, C_2>0$ such that 
\begin{align}
\norm{u(\cdot, t)}_{L^\infty(B(1))}
\leq C_0\frac{\ell(t)}{(-t)}
\,,
\qquad
\norm{\nabla^2u(\cdot,t)}_{L^\infty(B(1))}
\le C_2 \frac{1}{(-t)\ell(t)}
\,,
\label{eq:pointwise:velocity}
\end{align}
for all $t\in (-1,0)$.  Then, $(0,0)$ is a regular point.
\end{theorem}

\begin{remark}
\label{rem:pointwise:comments}
The following comments on the hypotheses of Theorem~\ref{thm:pointwise:variable} are in order:
\begin{itemize}[leftmargin=2em]
\item[(a)] For the power law $\ell(t)=(-t)^\gamma$ with $\gamma \in (0,1/2)$, the two Type~II bounds in~\eqref{eq:pointwise:velocity} say that the natural self-similar Euler rescaling $U(y,t) :=(-t)^{1-\gamma}u((-t)^\gamma y,t)$ satisfies $\norm{U(\cdot,t)}_{L^\infty}\leq C_0$ and $\norm{\nabla^2_y U(\cdot,t)}_{L^\infty}\leq C_2$, where the $L^\infty$ norm is taken on the expanding ball $B((-t)^{-\gamma})$.

\item[(b)] If the first bound in~\eqref{eq:pointwise:velocity} holds with the parabolic length scale $\ell(t) = (-t)^{1/2}$, then we have a standard Type~I bound. In that case, regularity is known due to the works of Chen, Strain, Tsai, and Yau~\cite{CSTY08,CSTY09}, Koch, Nadirashvili, Seregin, and \v{S}ver\'ak~\cite{KNSS09}, or Seregin and \v{S}ver\'ak~\cite{SereginSverak09}.

\item[(c)] The Type~II assumption made in~\eqref{eq:pointwise:velocity} is more restrictive than necessary, and the bounds on the second derivative can be relaxed. In Section~\ref{sec:sharp} we prove the same regularity result under weaker hypotheses: a H\"older seminorm bound on the azimuthal vorticity $\omega_\theta=\p_zu_r-\p_ru_z$ and a supremum bound for the potential vorticity $\Omega=\omega_\theta/r$, at the  Euler length. Precisely, in  Theorem~\ref{thm:sharp} we assume, for some $0<\alpha\le1$, the velocity $L^\infty$ bound in~\eqref{eq:pointwise:velocity} together with
\begin{align*}
[\omega_\theta(\cdot,t)]_{C^\alpha(B(1))}\le\frac{C_\alpha}{(-t) \ell(t)^\alpha}\,,\qquad
\norm{\Omega(\cdot,t)}_{L^\infty(B(1))}\le\frac{C_H}{(-t)\ell(t)}
\,,
\end{align*}
for some constants $C_\alpha, C_H>0$.
\item[(d)] The same weakening of regularity assumptions from Theorem~\ref{thm:sharp} also applies to the main result of~\cite[Section~4]{CIV26}. Theorem~4.5 in~\cite{CIV26} excludes nontrivial $C^2$ axisymmetric globally self-similar Euler profiles with similarity exponent $\gamma<1/2$ which obey the natural far-field bounds~\cite[(3.8)]{CIV26}. In Proposition~\ref{prop:civ} we extend this result to continuous self-similar profiles which vanish at the origin and are sublinear at infinity, whose azimuthal vorticity is locally H\"older continuous, and whose potential vorticity is locally bounded; no far-field decay beyond sublinearity is assumed.
\end{itemize}
\end{remark}

\subsection{Prior results}
\label{sec:prior:results}
In the absence of swirl, global regularity for the axisymmetric 3D Navier-Stokes equations was established by 
Ladyzhenskaya~\cite{Ladyzhenskaya68} and by Ukhovskii and Yudovich~\cite{UY68}. 

In the presence of swirl, the regularity of axisymmetric solutions to 3D Navier-Stokes has been established under the additional hypothesis of a Type~I bound, an assumption which is invariant under parabolic scaling. Chen, Strain, Yau, and Tsai~\cite{CSTY08} proved regularity of the solution at $t=0$ assuming the pointwise bound $|u|\leq C(r^2-t)^{-1/2}$. Chen, Strain, Tsai, and Yau~\cite{CSTY09} and Koch, Nadirashvili, Seregin, and \v{S}ver\'ak~\cite{KNSS09} proved regularity assuming that $r|u|\leq C$. Seregin and \v{S}ver\'ak~\cite{SereginSverak09} imposed a temporal Type~I bound on the meridional velocity alone:~$ |b|\leq C (-t)^{-1/2}$. Zhang~\cite{QiZhang26} recently proved a one-component one-sided Type~I result: if $u_r \geq - C (-t)^{-1/2}$ for some $C>0$, and if $r u_\theta$ is bounded, then the solution is regular. All of these proofs rely on the structure of the equations satisfied by the circulation and by the potential vorticity $\Omega=\omega_\theta/r$. We refer to the book of Tsai~\cite[Chapter 10]{Tsai18} for a detailed account of the many conditional regularity results for solutions of 3D axisymmetric Navier-Stokes.

Escauriaza, Seregin, and \v{S}ver\'ak~\cite[Theorem~1.4]{ESS03} proved the regularity of suitable weak solutions which are bounded in the scale-invariant norm $L^\infty_tL^3_x$; their argument is by contradiction and gives no quantitative bound. Tao~\cite{Tao21} made this result quantitative for classical solutions on $\RR^3$ and obtained a quantitative lower bound, of triple logarithmic type, for the blowup rate of the $L^3$ norm at a finite time singularity, along a sequence of times approaching that singularity. Subsequently, Palasek~\cite{Palasek21} replaced the triple logarithm by a double logarithm, for the weighted critical norms $\|r^{1-3/q}u\|_{L^q}$, both for axisymmetric solutions with $q\in(2,3]$ and for solutions without any symmetry with $q\in(3,\infty)$. O\.za\'nski and Palasek~\cite{OzanskiPalasek23} then bounded every derivative of an axisymmetric solution by a double exponential of its weak $L^3$ norm. Seregin~\cite{Seregin20} proved that at a singular axis point the three scale-invariant local energy quantities arising in the theory of partial regularity all have an infinite upper limit, so that an axisymmetric singularity must be Type~II also in the local-energy sense. Seregin also formulated Type~II blowup scenarios through such weighted scale-invariant quantities and studied them in~\cite{Seregin23,Seregin24,Seregin25,Seregin26}. 

The argument in this paper is inspired by the axisymmetric Euler profile theorem that we have established in our earlier work~\cite[Theorem~4.5]{CIV26}. We adapt the potential-vorticity transport along backward meridional trajectories, from the proof of~\cite[Theorem~4.5]{CIV26}, to ancient limits of the rescaled vorticity system; here the circulation is eliminated by the rescaling instead. The type of argument employed here, in which we zoom in at a sequence of points and times where a desired bound fails, and apply a rigidity statement to the resulting ancient limit, is the one through which Type~I singularities were excluded previously~\cite{KNSS09,SereginSverak09}. We emphasize that in this paper we do not assume a self-similar ansatz for the solution.  

\subsection{Main ideas of the proof}
\label{sec:intro:ideas}
The proof of Theorem~\ref{thm:pointwise:variable} has two parts. First, we show that the gradient of the meridional velocity becomes small compared with its scale-invariant bound, namely
\begin{align}
\lim_{t\uparrow0}(-t)\norm{\nabla b(\cdot,t)}_{L^\infty(B(\rho_0))}=0
\quad\mbox{on a fixed ball }B(\rho_0)\,,\quad 0<\rho_0<1
\,.
\label{eq:pointwise:meridional:small}
\end{align}
Second, we show that~\eqref{eq:pointwise:meridional:small} implies regularity at $(0,0)$.

For the first part we argue by contradiction. If~\eqref{eq:pointwise:meridional:small} fails, we zoom in at the length $\ell(t)$ along a sequence of times and points at which $(-t)|\nabla b|$ stays bounded away from zero, and we ask what the limit satisfies. In the context of Type~II weighted scale-invariant bounds on shrinking cylinders, this approach was considered by Seregin~\cite{Seregin23,Seregin24,Seregin25,Seregin26}. At a length larger than the parabolic one the pressure is not controlled after the zoom, so we  pass to the limit in the vorticity equation rather than in the momentum one. Two facts make the limit tractable. First, the circulation $\Gamma$ is bounded on a fixed ball, by partial regularity and the maximum principle for~\eqref{eq:swirl:eqn}. A bounded circulation is subcritical for the rescaling by an Euler length: after the zoom-in, it is bounded by a multiple of $(-t)/\ell(t)^2$, which tends to zero as $t\uparrow 0$. The swirl of every limit therefore vanishes, and with it the source $\p_z(u_\theta^2/r^2)$ of the potential vorticity equation~\eqref{eq:Omega:eqn} vanishes as well. For the ancient limits arising in his Type~II scenarios, Seregin~\cite[Section~2]{Seregin24} observed in the same way that the swirl vanishes. Since the swirl source vanishes, the limiting potential vorticity is transported by a locally bounded, divergence free field which is Lipschitz in space with constant $C_1/|\tau|$ (see~\eqref{eq:pointwise:meridional:bounded}); here $\tau \in (-\infty,-1]$ is the rescaled time of the ancient limit of the vorticity system. The second is that along the backward characteristics of that limiting divergence free field the potential vorticity is constant, and at time $\tau$ it inherits the bound $C/|\tau|$ from the hypotheses~\eqref{eq:pointwise:velocity}. Therefore the limiting azimuthal vorticity vanishes at $\tau=-1$: in the finite-axis case through the potential vorticity, and in the receding-axis case by transporting the azimuthal vorticity itself. This is the content of Lemmas~\ref{lem:ancient:axis} and~\ref{lem:ancient:receding}. The limiting meridional velocity at $\tau=-1$ is then curl-free, divergence-free, and bounded on $\RR^3$, hence constant. On the other hand, the selection of the contradiction sequence forces the gradient of the limiting meridional velocity to be nonzero at one point. This contradiction proves~\eqref{eq:pointwise:meridional:small}; see Proposition~\ref{prop:pointwise:expanding}.

For the second part, we apply the maximum principle in the swirl equation~\eqref{eq:swirl:eqn} to $(-t)^\eta u_\theta$ and obtain $|u_\theta|\leq C_\eta(-t)^{-\eta}$ on a smaller fixed ball, for every $\eta>0$. A localized enstrophy estimate then gives $u\in L^4_tL^6_x$ on a fixed parabolic cylinder with top $(0,0)$. Regularity then follows from the localized version of the Ladyzhenskaya-Prodi-Serrin criterion; see Section~\ref{sec:pointwise:regularity}.

\begin{remark}[\textbf{What is special about axisymmetry?}]
\label{rem:intro:axisymmetry}
The exclusion of Type~I singularities is not known without the assumption of axisymmetry, and we do not know how to prove Theorem~\ref{thm:pointwise:variable} without it when $\ell(t)=(-t)^\gamma$ with $2/5\leq\gamma<1/2$, or when $\ell$ is a logarithmic enlargement of the parabolic length. What this rotational invariance supplies is not merely the reduction to the variables $(r,z)$, which leaves vortex stretching intact, but the two identities~\eqref{eq:swirl:eqn} and~\eqref{eq:Omega:eqn}. The circulation obeys a drift-diffusion equation with no zeroth order term, so that it satisfies a maximum principle. Moreover, the circulation scales like viscosity: an axisymmetric solution carries an a priori bounded critical quantity, and we know of none for~\eqref{eq:nse} in the general setting of flows on $\RR^3$. The potential vorticity obeys a transport-diffusion equation whose only other term is the swirl source $\p_z(\Gamma^2/r^4)$; without swirl this is the global regularity of~\cite{Ladyzhenskaya68,UY68}, and with swirl the Liouville theorems of~\cite{KNSS09} are proved through the same two equations. 
\end{remark}

\begin{remark}[\textbf{Euler profiles and the viscous question}]
\label{sec:intro:euler}
Consider a putative 3D Navier-Stokes singularity in the absence of external forces, constructed from a self-similar Euler singularity by treating the viscosity as an asymptotic perturbation. With the singular time at $t=0$, the similarity ansatz $u(x,t)=(-t)^{\gamma-1}U(x/(-t)^\gamma,-\log(-t))$ carries the viscosity coefficient $(-t)^{1-2\gamma}$, which tends to zero as $t\uparrow 0$ precisely when $\gamma<1/2$. These are the power laws of Definition~\ref{def:Euler:length}. In essence, Theorem~\ref{thm:pointwise:variable} shows what such a construction has to give up in the axisymmetric setting: the viscous solution cannot remain bounded in $C^2$ in similarity variables, on a fixed ball in the original variables, no matter what profile it is meant to approach at the blowup time.
\end{remark}

\subsection*{Organization of the paper}
In Section~\ref{sec:pointwise} we record the axisymmetric identities, a regular annulus obtained from partial regularity, and a transport lemma. In Section~\ref{sec:pointwise:ancient} we prove the two ancient-limit lemmas and deduce Proposition~\ref{prop:pointwise:expanding} from them. In Section~\ref{sec:pointwise:regularity} we give the local regularity conclusion and the logarithmic examples. In Sections~\ref{sec:sharp} and~\ref{sec:sharp:civ} we prove the regularity statements for Navier-Stokes solutions and for Euler self-similar profiles under weaker assumptions. 

\section{Axisymmetric identities and auxiliary results}
\label{sec:pointwise}
In this section we record the axisymmetric identities (see e.g.~\cite[Chapter 10]{Tsai18}) and a few technical lemmas which will be used later in the proof.
We decompose the velocity field into its meridional and swirl components
\begin{equation}
\label{eq:axi:b:w:def}
u=b+w\,,
\qquad 
b:=u_r e_r+u_ze_z\,,
\qquad 
w:=u_\theta e_\theta\,,
\end{equation}
and note that both $b$ and $w$ are divergence free. 
For the azimuthal component of the vorticity we denote
\begin{equation*}
\omega_\theta :=(\curl u)_\theta = \p_z u_r - \p_r u_z\,,
\qquad 
\Omega:=\frac{\omega_\theta}{r}
\,,
\end{equation*}
and note that 
\[
\curl b=\omega_\theta e_\theta\,.
\]
We refer to $\Omega$ as the \emph{potential vorticity}: denoting $\omega=\curl u$ we have $\Omega=\omega\cdot\nabla\theta$, which is Ertel's potential vorticity associated with the azimuthal angle $\theta$, and $\theta$ is materially conserved precisely when the swirl vanishes.

For a vector field $F$ we write $|\nabla F|$ to denote the Euclidean norm of the Jacobian matrix, so that $|\curl F|\le\sqrt2\,|\nabla F|$.  Lipschitz and H\"older seminorms of rotation-invariant scalars are taken in the Euclidean distance. 

\subsection{Axisymmetric identities}
\label{sec:pointwise:identities}
For smooth axisymmetric solutions, the meridional projection of~\eqref{eq:nse} is given by
\begin{subequations}
\label{eq:meridional:nse}
\begin{align}
\p_t b+\div(b\otimes b)+\nabla\pi &=\Delta b-\div(w\otimes w)
\,,
\\
\div b&=0 \,,
\end{align}
with the swirl velocity contributing the Reynolds stress
\begin{align}
\div(w\otimes w) =-\frac{u_\theta^2}{r}\,e_r
\,.
\end{align}
\end{subequations}
The swirl velocity $u_\theta$ and the circulation $\Gamma=ru_\theta$ satisfy
\begin{align}
\p_tu_\theta+b\cdot\nabla u_\theta+\frac{u_r}{r}u_\theta
=\Big(\Delta-\frac1{r^2}\Big)u_\theta\,,
\qquad
\p_t\Gamma+b\cdot\nabla\Gamma
=\Big(\Delta-\frac2r\p_r\Big) \Gamma
\,.
\label{eq:swirl:eqn}
\end{align}
The azimuthal vorticity obeys
\begin{align}
\p_t\omega_\theta+b\cdot\nabla\omega_\theta
= \Big(\Delta-\frac1{r^2}\Big)\omega_\theta + \frac{u_r}{r}\,\omega_\theta+\frac1r\,\p_z\big(u_\theta^2\big)
\,.
\label{eq:vort:theta}
\end{align}
Dividing~\eqref{eq:vort:theta} by $r$, and using $r^{-1}(\Delta-r^{-2})(rg)=(\Delta+2r^{-1}\p_r)g$ for axisymmetric scalars $g$, we obtain
\begin{align}
\p_t\Omega 
+b\cdot\nabla\Omega
=\Big(\Delta+\frac2r\,\p_r\Big)\Omega
+\p_z\Big(\frac{u_\theta^2}{r^2}\Big)
\,.
\label{eq:Omega:eqn}
\end{align}

We recall the standard fact that near the axis of symmetry, the components of a smooth axisymmetric vector field $u$ take the form $u_r=r\,a(r^2,z,t)$, $u_\theta=r\,c(r^2,z,t)$, $u_z=d(r^2,z,t)$ for some smooth functions $a,c,d$; thus, $\Omega$ and $u_\theta^2/r^2$ extend smoothly to the axis, and~\eqref{eq:Omega:eqn} holds on the whole domain. 

The following lemma records bounds for $\Omega$ in terms of the Hessian of $u$.

\begin{lemma}[{\bf Pointwise bounds across the axis}]
\label{lem:pointwise}
Let $u$ be a smooth axisymmetric vector field, not necessarily a solution of~\eqref{eq:nse}, defined on an open axisymmetric set $\Ucal\subset\RR^3$ which meets the axis of symmetry.
\begin{enumerate}[label=(\alph*),leftmargin=2em]

\item We have $\big|\omega_\theta/r\big|\le\sqrt2\,|\nabla^2u|$ at every point of $\Ucal$ off the axis, where $\omega_\theta =\p_z u_r-\p_r u_z$ and $|\nabla^2 u|$ denotes the Euclidean norm of the full tensor $(\p_i\p_ju_k)_{i,j,k}$.

\item With meridional part $b=u_r e_r+ u_z e_z$ and swirl part $w=u_\theta e_\theta$, we have $|\nabla^j u|^2=|\nabla^j b|^2+|\nabla^j w|^2$ pointwise in $\Ucal$ for $j=0,1,2$, and $|\nabla\omega_\theta|\le\sqrt2\,|\nabla^2 u|$ off the axis. In particular, $\omega_\theta$, extended by zero to the axis, is Lipschitz on every open convex set $K \subseteq \Ucal$, with constant $\sqrt{2}\| \nabla^2 u\|_{L^\infty(K)}$.
\end{enumerate}
\end{lemma}

\begin{proof}[Proof of Lemma~\ref{lem:pointwise}]
Every quantity in the lemma is invariant under rotations about the axis; thus, it suffices to work at a point $x_0=(x_1,0,z)\in \Ucal$, with $x_1 = r >0$. In Cartesian components, an axisymmetric field $F=F_re_r+F_\theta e_\theta+F_ze_z$, with $F_r,F_\theta,F_z$ functions of $(r,z)$, reads
\begin{align}
F_1=F_r\,\frac{x_1}{r}-F_\theta\,\frac{x_2}{r}\,,\qquad
F_2=F_r\,\frac{x_2}{r}+F_\theta\,\frac{x_1}{r}\,,\qquad
F_z=F_z(r,z)
\,.
\label{eq:pointwise:cartesian}
\end{align}
Differentiating~\eqref{eq:pointwise:cartesian} and evaluating at $x_0$, we obtain
\begin{align}
\p_2F_1(x_0)=-\frac{F_\theta}{r}(x_0)
\,.
\label{eq:pointwise:axis:identity}
\end{align}
In order to prove item~(a), we apply~\eqref{eq:pointwise:axis:identity} to the smooth axisymmetric field $F=\omega:=\curl u$. On the half-plane $\{x_2=0,\,x_1>0\}$, where $u_1=u_r$ and $r=x_1$, by~\eqref{eq:pointwise:cartesian} we have $\omega_2=\p_zu_1-\p_1u_z=\omega_\theta$. Since $\omega_1=\p_2u_z-\p_zu_2$, it follows that $\omega_\theta/r = -\p_2\omega_1(x_0)=\p_2\p_zu_2-\p_2^2u_z$, and hence 
\[
\Big|\frac{\omega_\theta}{r}\Big|\le\sqrt2\big(|\p_2\p_zu_2|^2+|\p_2^2u_z|^2\big)^{1/2}\le\sqrt2\,|\nabla^2u|\,,
\]
because the two derivatives are distinct entries of the tensor $(\p_i\p_ju_k)_{i,j,k}$.

In order to prove item~(b), let $S$ be the reflection $(x_1,x_2,z)\mapsto(x_1,-x_2,z)$, which preserves $\Ucal$. By~\eqref{eq:pointwise:cartesian}, the Cartesian components $b=(u_r \frac{x_1}{r},\,u_r \frac{x_2}{r},\,u_z)$ and $w=(-u_\theta \frac{x_2}{r},\,u_\theta \frac{x_1}{r},\,0)$ satisfy $b(Sx)=Sb(x)$ and $w(Sx)=-Sw(x)$ off the axis. Hence $b=\frac12\big(u+S\,u(S\,\cdot)\big)$ and $w=\frac12\big(u-S\,u(S\,\cdot)\big)$ extend smoothly to $\Ucal$, and $b_1$, $b_z$, $w_2$ are even in $x_2$, while $b_2$, $w_1$ are odd. On $\{x_2=0\}$, a derivative of an even function vanishes when $\p_2$ occurs in it an odd number of times, and a derivative of an odd function vanishes when $\p_2$ occurs an even number of times. Thus, for each component $k$ and each $j\le2$, the entries of $\nabla^jb_k$ and of $\nabla^jw_k$ which do not vanish at $x_0$ occupy disjoint index sets. Therefore, when contracting the tensors we obtain $\nabla^jb:\nabla^jw=0$ at $x_0$, and thus $|\nabla^ju|^2=|\nabla^jb|^2+|\nabla^jw|^2$ there. By rotation invariance the identity holds at every point off the axis, and on the axis by continuity.

We recall that $\omega_\theta=\p_zu_1-\p_1u_z$ on the half-plane $\{x_2=0,\,x_1>0\}$. Differentiating this identity along the half-plane, and evaluating at a point $x_0 = (x_1,0,z)$ with $x_1>0$, we have that $\p_1\omega_\theta=\p_1\p_zu_1-\p_1^2u_z$ and $\p_z\omega_\theta=\p_z^2u_1-\p_1\p_zu_z$, while $\p_2\omega_\theta =0$ because $\omega_\theta$ is constant on the circle about the axis through $x_0$, to which $e_2$ is tangent. Thus $|\nabla\omega_\theta(x_0)|^2\le2\big(|\p_1\p_zu_1|^2+|\p_1^2u_z|^2+|\p_z^2u_1|^2+|\p_1\p_zu_z|^2\big)\le2|\nabla^2u(x_0)|^2$. The bound holds off the axis by rotation invariance.

Lastly, by item~(a) we have $|\omega_\theta|\le\sqrt2\,r\,|\nabla^2u|$ off the axis, so that the extension of $\omega_\theta$ by zero to the axis is continuous on $\Ucal$. Let $K\subset\Ucal$ be open and convex, let $x,y\in K$, and let $M=\sqrt2\|\nabla^2u\|_{L^\infty(K)}$. If the segment $[x,y]$ lies in the axis, both values vanish. Otherwise the segment meets the axis in at most one point, $\omega_\theta$ is continuous on $[x,y]$, and $C^1$ with $|\nabla\omega_\theta|\le M$ off that axis point. The desired bound follows by the mean value theorem, applied on each side of that point if there is one.
\end{proof}

\subsection{A regular annulus and bounded circulation}
\label{sec:pointwise:annulus}

By partial regularity, there exists a sphere about the origin which contains no terminal singular point, and on the ball bounded by this sphere the maximum principle bounds the circulation $\Gamma = r u_\theta$. These statements are used in the proof as follows: the bounded circulation is what makes the swirl of every limit vanish in Section~\ref{sec:pointwise:ancient}, and the regular annulus is where the cutoff errors of the enstrophy estimate of Section~\ref{sec:pointwise:regularity} are supported.

\begin{lemma}[{\bf A regular annulus and bounded local circulation}]
\label{lem:pointwise:local:circulation}
For every $0\leq R_1<R_0<1$ there exist $R_1<R_-<R_+<R_0$ and $t_0<0$
such that $u$ and all its spatial derivatives are bounded on the cylindrical annulus domain 
$\{R_-<|x|<R_+\}\times(t_0,0)$. For any fixed
$R_*\in(R_-,R_+)$ there is a constant $C_\Gamma<\infty$ such that 
\begin{align}
\norm{\Gamma}_{L^\infty(B(R_*)\times(t_0,0))}\leq C_\Gamma
\,.
\label{eq:pointwise:circulation}
\end{align}
\end{lemma}

\begin{proof}[Proof of Lemma~\ref{lem:pointwise:local:circulation}]
Let $S_0\subset\overline{B(R_0)}$ be the set of $x$ for which $(x,0)$ is a singular point. By partial regularity~\cite{CKN,Lin98,LadyzhenskayaSeregin99}, $S_0$ has zero one-dimensional Hausdorff measure, so that its image under the Lipschitz map $x\mapsto|x|$ has zero Lebesgue measure and we may choose $R\in(R_1,R_0)$ for which $\partial B(R)$ carries no point of $S_0$. Every point $(x,0)$ with $x\in\partial B(R)$ is regular, so that by compactness of $\partial B(R)$ there are $\delta>0$, with $R_1<R-\delta<R+\delta<R_0$, and $t_0<0$ such that $u$ is bounded on $\{R-\delta<|x|<R+\delta\}\times(t_0,0)$. We fix $R-\delta<R_-<R_+<R+\delta$ and bound all spatial derivatives of $u$ on $\{R_-<|x|<R_+\}\times(t_0,0)$, uniformly up to time zero, after possibly increasing $t_0$. For this purpose, we apply the $\eps$-regularity theorem in the form stated by Seregin and \v{S}ver\'ak~\cite[Lemma~1.2]{SereginSverak05} (see also~\cite[Lemma~2.2 and Remark~2.3]{ESS03}), which asks that the two scale-invariant quantities
\[
C_\zeta(\rho):= \frac1{\rho^2}\int_{Q(\zeta,\rho)}|u|^3\,dx\,dt
\,,
\qquad
D_\zeta(\rho):=\frac1{\rho^2}\int_{Q(\zeta,\rho)}|\pi|^{3/2}\,dx\,dt
\,,
\]
add up to less than a universal constant $\eps_*>0$ at a single scale $\rho$; here we denote spacetime cylinders by $Q(\zeta,\rho):=B(x,\rho)\times(s-\rho^2,s)$, for $\zeta=(x,s)$. Two scales enter, and the order in which they are fixed is what makes the argument work: first $\rho_0$, using only the supremum of $|u|$ on the annulus, and then $\rho<\rho_0$, at the value of $D_\zeta(\rho_0)$ that the first choice determines.

Let $M:= \sup_{\{R-\delta<|x|<R+\delta\}\times(t_0,0)} |u|$, and let $\rho_0\in(0,1]$, with $\rho_0^2<-t_0$, be so small that $Q(\zeta,\rho_0)\subset\{R-\delta<|x|<R+\delta\}\times(t_0,0)$ for every $\zeta=(x,s)$ with $R_-\le|x|\le R_+$ and $t_0+\rho_0^2<s<0$. For such $\zeta$ and $0<\rho\le\rho_0$ the cylinder $Q(\zeta,\rho)$ lies in the region where $|u|\le M$, so that $C_\zeta(\rho)\le\frac{4\pi}3M^3\rho_0^3$. Using the pressure decay estimate of Seregin and \v{S}ver\'ak~\cite[Lemma~2.1]{SereginSverak05} (with $\alpha=\frac12$, and after a parabolic rescaling of $Q(\zeta,\rho_0)$ to a unit cylinder), we obtain that
$D_\zeta(\rho)\le c (\frac{\rho}{\rho_0})^{1/2} D_\zeta(\rho_0) +c M^3\rho_0^3$
for a universal constant $c>0$. Neither of the two terms $\frac{4\pi}3M^3\rho_0^3$ and $cM^3\rho_0^3$ depends on $\rho$, so that both are made small by the choice of $\rho_0$ alone, which depends only on $M$, on $c$, and on $\eps_*$; this is the first choice, and it fixes $\rho_0$. Once $\rho_0$ is fixed, $D_\zeta(\rho_0)$ is bounded uniformly in $\zeta$, since $D_\zeta(\rho_0)\le\rho_0^{-2}\norm{\pi}_{L^{3/2}(Q)}^{3/2}$, and the one term which remains, $c(\frac{\rho}{\rho_0})^{1/2}D_\zeta(\rho_0)$, is made small by letting $\rho<\rho_0$ be sufficiently small; this is the second choice. The two cannot be made in the other order, because the bound on $D_\zeta(\rho_0)$ degenerates as $\rho_0\to0$. We thus obtain $C_\zeta(\rho)+D_\zeta(\rho)<\eps_*$ for every $\zeta=(x,s)$ with $R_-\le|x|\le R_+$ and $t_0+\rho_0^2<s<0$. Therefore, by $\eps$-regularity we have $|\nabla^k u(x,s)|\le c_k\rho^{-k-1}$ for all $k\ge0$ with universal constants $c_k$, for all $R_-\le|x|\le R_+$, and all $t_0+\rho_0^2<s<0$; replacing $t_0$ by $t_0+\rho_0^2$ gives the desired bounds on the cylindrical annulus.

We fix $R_*\in(R_-,R_+)$. By~\eqref{eq:swirl:eqn}, the circulation $\Gamma$ satisfies a drift-diffusion equation with no zeroth order term, and it vanishes on the axis. For every $s<0$, the maximum principle on $\overline{B(R_*)}\times[t_0,s]$, where $\Gamma$ is smooth, bounds $|\Gamma|$ by its values at time $t_0$ and on $\partial B(R_*)\times(t_0,s)$. The former are bounded since $u$ is smooth at time $t_0$, the latter by $R_*$ times the bound on $u$ in the annulus; neither bound depends on $s$. Although the coefficient $2r^{-1}$ of $\p_r\Gamma$ in~\eqref{eq:swirl:eqn} is singular on the axis, $\Gamma$ vanishes there, so that a positive maximum or a negative minimum of $\Gamma$ which is not attained on the parabolic boundary is attained off the axis, where the classical argument applies. Letting $s\uparrow0$, we obtain~\eqref{eq:pointwise:circulation}.
\end{proof}

\subsection{Transport} 
In our proof we argue that certain limits of the vorticity obtained by zooming in solve in the weak sense a transport equation whose drift is itself only a weak-$*$ limit. For this purpose, we recall that for a divergence free drift which is bounded and Lipschitz in space on compact sets, with a Lipschitz constant which is locally integrable in time, a continuous distributional solution of the transport equation is constant along the characteristics of the drift. We state here a version adapted to our setting:
\begin{lemma}[{\bf Constancy along characteristics}]
\label{lem:pointwise:characteristics}
Let $\Ucal\subset\RR^3\times(-\infty,-1]$ be relatively open. Let $B\colon\Ucal\to\RR^3$ be bounded on compact sets, measurable in $\tau$, Lipschitz in $y$ on compact sets with a Lipschitz constant $\lambda(\tau)$ which is locally integrable in time, and divergence free in $y$ for almost every $\tau$. Let $\Lambda\colon\Ucal\to\RR$ be a continuous \emph{weak solution} of the transport equation $\partial_\tau \Lambda + \div_{\!\! y} ( B \Lambda) = 0$ in the \emph{interior} of $\Ucal$.\footnote{That is, 
$\int_{\Ucal}\Lambda\big(\p_\tau\varphi+B\cdot\nabla\varphi\big)\,dy\,d\tau=0$ for every $\varphi\in C^\infty_0(\mathrm{int}\,\Ucal)$.}
Let $X\colon(-\infty,-1]\to\RR^3$ be absolutely continuous on compact subintervals, with $\dot X(\tau)=B(X(\tau),\tau)$ for almost every $\tau$, and suppose that $\{(X(\tau),\tau)\colon\tau_0\le\tau\le-1\}$ is a compact subset of $\Ucal$ for every $\tau_0<-1$. Then $\tau\mapsto\Lambda(X(\tau),\tau)$ is constant on $(-\infty,-1]$.
\end{lemma}
The proof of the above result is standard: with the notation of Lemma~\ref{lem:pointwise:characteristics}, it mollifies the solution in space, bounds the commutator of DiPerna and Lions~\cite[Lemma~II.1]{DiPernaLions89} uniformly on compact sets by a constant multiple of $\lambda(\tau)$ times the modulus of continuity of $\Lambda$ at the mollification scale (this uses the continuity of $\Lambda$), and integrates the mollified equation in time along $X$. This commutator bound is written out, in a stationary setting, in Step~3 of the proof of Proposition~\ref{prop:civ}. We omit further details.
 
\section{Ancient limits at the Euler length}
\label{sec:pointwise:ancient}
In this section we construct ancient limits at the Euler length, and  use them to prove that on a fixed ball the meridional gradient becomes small compared with the scale-invariant bound $C_1/(-t)$, which our hypotheses give directly, see~\eqref{eq:pointwise:meridional:bounded} below. In Section~\ref{sec:pointwise:regularity} we show that this smallness implies regularity, and thus Theorem~\ref{thm:pointwise:variable}.

\begin{proposition}[{\bf Smallness of the meridional gradient}]
\label{prop:pointwise:expanding}
Let $\ell$ be an Euler length, in the sense of Definition~\ref{def:Euler:length}, and assume that the bounds~\eqref{eq:pointwise:velocity} hold for all $t\in(-1,0)$. Then
\begin{align}
\lim_{t\uparrow0}(-t)\norm{\nabla b(\cdot,t)}_{L^\infty(B(1/2))}=0
\,.
\label{eq:pointwise:expanding:conclusion}
\end{align}
\end{proposition}

We prove Proposition~\ref{prop:pointwise:expanding} by contradiction. If~\eqref{eq:pointwise:expanding:conclusion} fails, there are times $t_n\uparrow0$ and points $x_n\in B(1/2)$ at which $(-t_n)|\nabla b(x_n,t_n)|$ stays bounded away from zero. We zoom in at time $t_n$ about a base point built from $x_n$, in space using the Euler length $\ell(t_n)$, and in time using $|t_n|$ itself. Our goal is to analyze the ancient limits of the rescaled solution. A priori these limits are not known to solve the Euler equations in their momentum form: at a length larger than the parabolic one the pressure is not controlled after the zoom-in, so that we cannot pass to the limit in the momentum equation. Instead, we pass to the limit in the  vorticity equation.

Two facts make analyzing the limit tractable, and both were described in Section~\ref{sec:intro:ideas}: by Lemma~\ref{lem:pointwise:local:circulation} the circulation is bounded on a fixed ball, and a bounded circulation is subcritical for this rescaling by~\eqref{visczero}, so that the swirl of every limit vanishes together with the swirl source $\p_z(u_\theta^2/r^2)$ of~\eqref{eq:Omega:eqn}; and Lemma~\ref{lem:pointwise:characteristics} makes the limiting vorticity constant along the backward characteristics of the limiting drift, so that the bound $C/|\tau|$ it inherits from~\eqref{eq:pointwise:velocity} forces it to vanish. The two eliminations, first of the circulation and then of the potential vorticity, each along backward meridional trajectories, are the main idea of the proof of~\cite[Theorem~4.5]{CIV26}. Here the circulation is eliminated instead by the rescaling, as observed for the ancient limits of an Euler scaling by Seregin~\cite[Section~2]{Seregin24}. What is new here is that the potential vorticity is transported along the trajectories of a time-dependent field which is only a weak-$*$ limit, for which the transport equation holds only weakly and only off the axis (we keep these trajectories off the axis through a quantitative lower bound on their distance to it), and that the potential vorticity vanishes because of the bound $C/|\tau|$ inherited from~\eqref{eq:pointwise:velocity}, rather than because of an exponential weight as in~\cite{CIV26}.

The main result of the construction is the pair of Lemmas~\ref{lem:ancient:axis} and~\ref{lem:ancient:receding}: along a subsequence, the rescaled azimuthal vorticity converges locally uniformly to a limit which vanishes at $\tau=-1$, and the rescaled meridional velocity at $\tau=-1$ converges in $C^1$ on compact sets to a field which is curl free, divergence free, and bounded on all of $\RR^3$. Thus each component is bounded and harmonic, hence constant by Liouville's theorem. The selection of the points $x_n$ forces the gradient of this limit to be bounded away from zero at one point, which provides the desired contradiction. In Section~\ref{sec:pointwise:zoom} we define the zoom-in and record the bounds it inherits from~\eqref{eq:pointwise:velocity}, in Section~\ref{sec:pointwise:limits} we prove the two lemmas, and in Section~\ref{sec:pointwise:applications} we select the sequences $t_n$ and $x_n$ and complete the proof of Proposition~\ref{prop:pointwise:expanding}. 

\subsection{The rescaled fields and bounds uniform in the zoom-in}
\label{sec:pointwise:zoom}

Throughout this subsection we assume that $\ell$ is an Euler length and that the bounds in~\eqref{eq:pointwise:velocity} hold at every $t\in(-1,0)$. We take $T_1\in(0,1]$ so small that
\begin{align}
\ell(t)\le\frac{1}{4}
\qquad(-T_1<t<0)
\,.
\label{eq:pointwise:T:one}
\end{align}
Applying Taylor's formula along segments of length $\ell(t)$ inside $B(1)$, we deduce from~\eqref{eq:pointwise:velocity},~\eqref{eq:pointwise:T:one}, and Lemma~\ref{lem:pointwise}(b) that for $t \in (-T_1,0)$ 
\begin{align}
\norm{\nabla b(\cdot,t)}_{L^\infty(B(3/4))}\le\frac{C_1}{(-t)}
\,,\qquad
C_1:=\sqrt3\Big(2C_0+\frac{C_2}2\Big)
\,.
\label{eq:pointwise:meridional:bounded}
\end{align}
Note that the length scale $\ell$ cancels in this balance, so that the two hypotheses in~\eqref{eq:pointwise:velocity} interpolate to a gradient of the scale-invariant size $1/(-t)$, whatever $\ell$ is. Lemma~\ref{lem:pointwise}(a) in addition gives
\begin{align}
\norm{\Omega(\cdot,t)}_{L^\infty(B(1))}\le\frac{\sqrt2\,C_2}{(-t)\ell(t)}
\,.
\label{eq:pointwise:quotient}
\end{align}

We consider a sequence $t_n\uparrow0$ with $t_n\in(-T_1,0)$, and we let 
\[
\ell_n:=\ell(t_n) \,, \qquad \eps_n:=\frac{(-t_n)}{\ell_n^2} \,, 
\]
so that $\eps_n\to0$ by~\eqref{visczero}. We also let $a_n\in\RR^3$ be such that $|a_n|<\frac12$; we call $a_n$ the \emph{base point} of the zoom, since~\eqref{eq:pointwise:fixed:zoom} below rescales about the  point $a_n$. The sequences $\{t_n\}_{n\geq 1}$ and $\{a_n\}_{n\geq 1}$ are arbitrary both here and in Section~\ref{sec:pointwise:limits}; only in the proof of Proposition~\ref{prop:pointwise:expanding} do we build them from times and points at which the contradiction ansatz~\eqref{eq:pointwise:selected} holds. 

Upon recalling the decomposition $u = b+w$ from~\eqref{eq:axi:b:w:def}, we define
\begin{subequations}
\label{eq:pointwise:fixed:zoom}
\begin{align}
B_n(y,\tau)&=\frac{(-t_n)}{\ell_n}\,b\big(a_n+\ell_ny,\,t_n|\tau|\big)\,,
\\
Z_n(y,\tau)&=\frac{(-t_n)}{\ell_n}\,w\big(a_n+\ell_ny,\,t_n|\tau|\big)\,,
\\
P_n(y,\tau)&=\frac{(-t_n)^2}{\ell_n^2}\,\pi\big(a_n+\ell_ny,\,t_n|\tau|\big)
\,,\\
\Theta_n(y,\tau)&=(-t_n)\,\omega_\theta\big(a_n+\ell_ny,t_n|\tau|\big)
\,,
\label{eq:pointwise:Omega:n:def}
\end{align}
\end{subequations}
for all $\tau<0$ with $(-t_n)|\tau|<1$ and all $y\in\Dcal_n$, where
\begin{align*}
\Dcal_n:=\{y\colon a_n+\ell_ny\in B(1)\}=B\big(-\tfrac{a_n}{\ell_n},\tfrac{1}{\ell_n}\big)
\,.
\end{align*}
These are exactly the $(y,\tau)$ for which $a_n+\ell_ny\in B(1)$ and $t_n|\tau|\in(-1,0)$, so that the right sides of~\eqref{eq:pointwise:fixed:zoom} are well-defined. We also introduce a smaller ball $\Dcal_n^\flat\subset\Dcal_n$, defined as 
\begin{equation*}
\Dcal_n^\flat:=\{y\colon a_n+\ell_ny\in B(3/4)\}=B\big(-\tfrac{a_n}{\ell_n},\tfrac{3}{4\ell_n}\big) \,.
\end{equation*}
See Figure~\ref{fig:pointwise:domains}.
This smaller ball plays no role in definition~\eqref{eq:pointwise:fixed:zoom}; instead, this is the set on which the gradient bound~\eqref{eq:pointwise:meridional:bounded} is available after the zoom (see~\eqref{eq:pointwise:zoom:bounds:c} below). Since $|a_n|<\frac12$, we have $\Dcal_n^\flat\supset B(\rho_n)$ with $\rho_n:=\frac{1}{4\ell_n}\to\infty$, so that every compact subset of $\RR^3$ is contained in $\Dcal_n^\flat$ for all $n$ sufficiently large.
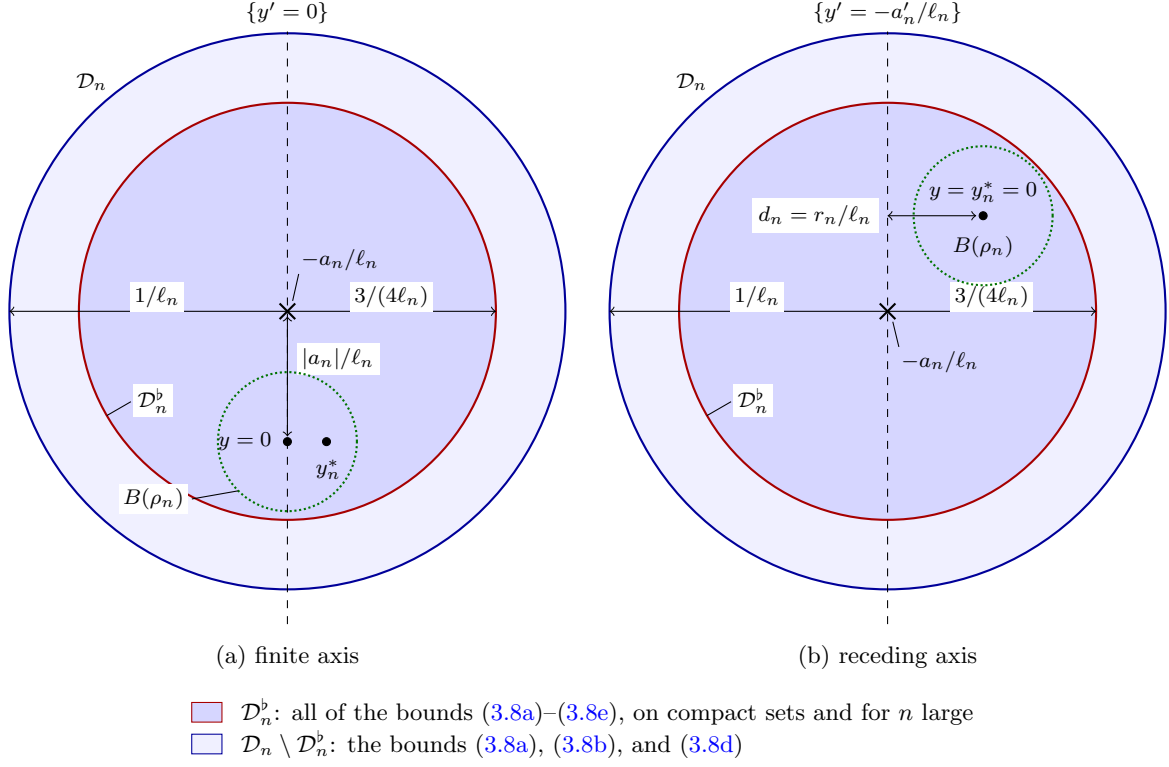
\begin{figure}[htb!]
\centering
\begin{tikzpicture}[scale=1.15]
%%% (a) the finite-axis case: $a_n$ is the foot of $x_n$ on the axis, so $y=0$ lies on the rescaled axis
\fill[blue!6]  (0,0) circle (3.2);
\fill[blue!16] (0,0) circle (2.4);
\draw[blue!60!black,thick] (0,0) circle (3.2);
\draw[red!65!black,thick]  (0,0) circle (2.4);
\draw[black,dashed] (0,-3.6) -- (0,3.22);
\node[font=\scriptsize,anchor=south,inner sep=1.5pt] at (0,3.26) {$\{y'=0\}$};
\node[font=\scriptsize,anchor=south east,inner sep=1.5pt] at (-2.02,2.49) {$\Dcal_n$};
\draw[thin] (210:2.4) -- (210:2.08);
\node[font=\scriptsize,anchor=west,inner sep=1.5pt,fill=white] at (210:2.05) {$\Dcal_n^\flat$};
\draw[->,thin] (0,0) -- (-3.2,0);
\node[font=\scriptsize,inner sep=1.5pt,fill=white] at (-1.5,0.17) {$1/\ell_n$};
\draw[->,thin] (0,0) -- (2.4,0);
\node[font=\scriptsize,inner sep=1.5pt,fill=white] at (1.2,0.17) {$3/(4\ell_n)$};
\draw[thick] (-0.09,-0.09) -- (0.09,0.09);
\draw[thick] (-0.09,0.09) -- (0.09,-0.09);
\draw[thin] (0.13,0.43) -- (0.06,0.12);
\node[font=\scriptsize,anchor=south west,inner sep=1.5pt] at (0.11,0.42) {$-a_n/\ell_n$};
\draw[<->,thin] (0,-0.06) -- (0,-1.44);
\node[font=\scriptsize,anchor=west,inner sep=2pt,fill=white] at (0.10,-0.55) {$|a_n|/\ell_n$};
\draw[green!45!black,thick,densely dotted] (0,-1.5) circle (0.8);
\draw[thin] (-0.566,-2.066) -- (-1.10,-2.16);
\node[font=\scriptsize,anchor=east,inner sep=1.5pt,fill=white] at (-1.14,-2.16) {$B(\rho_n)$};
\filldraw (0,-1.5) circle (1.3pt);
\node[font=\scriptsize,anchor=east,inner sep=2.5pt] at (-0.10,-1.5) {$y=0$};
\filldraw (0.45,-1.5) circle (1.3pt);
\node[font=\scriptsize,anchor=north,inner sep=2.5pt] at (0.47,-1.62) {$y_n^*$};
\node[font=\footnotesize,anchor=north] at (0,-3.72) {(a) finite axis};
\end{tikzpicture}
\hspace{0.3cm}
\begin{tikzpicture}[scale=1.15]
%%% (b) the receding-axis case: $a_n=x_n$, so $y=0$ is the image of the selected point
\fill[blue!6]  (0,0) circle (3.2);
\fill[blue!16] (0,0) circle (2.4);
\draw[blue!60!black,thick] (0,0) circle (3.2);
\draw[red!65!black,thick]  (0,0) circle (2.4);
\draw[black,dashed] (0,-3.6) -- (0,3.22);
\node[font=\scriptsize,anchor=south,inner sep=1.5pt] at (0,3.26) {$\{y'=-a_n'/\ell_n\}$};
\node[font=\scriptsize,anchor=south east,inner sep=1.5pt] at (-2.02,2.49) {$\Dcal_n$};
\draw[thin] (210:2.4) -- (210:2.08);
\node[font=\scriptsize,anchor=west,inner sep=1.5pt,fill=white] at (210:2.05) {$\Dcal_n^\flat$};
\draw[->,thin] (0,0) -- (-3.2,0);
\node[font=\scriptsize,inner sep=1.5pt,fill=white] at (-1.5,0.17) {$1/\ell_n$};
\draw[->,thin] (0,0) -- (2.4,0);
\node[font=\scriptsize,inner sep=1.5pt,fill=white] at (1.2,0.17) {$3/(4\ell_n)$};
\draw[thick] (-0.09,-0.09) -- (0.09,0.09);
\draw[thick] (-0.09,0.09) -- (0.09,-0.09);
\draw[thin] (0.13,-0.41) -- (0.06,-0.12);
\node[font=\scriptsize,anchor=north west,inner sep=1.5pt] at (0.11,-0.40) {$-a_n/\ell_n$};
\draw[<->,thin] (0,1.1) -- (1.02,1.1);
\node[font=\scriptsize,anchor=east,inner sep=2.5pt,fill=white] at (-0.10,1.1) {$d_n=r_n/\ell_n$};
\draw[green!45!black,thick,densely dotted] (1.1,1.1) circle (0.8);
\node[font=\scriptsize,anchor=north,inner sep=2.5pt] at (1.1,0.95) {$B(\rho_n)$};
\filldraw (1.1,1.1) circle (1.3pt);
\node[font=\scriptsize,anchor=south,inner sep=2.5pt] at (1.1,1.18) {$y=y_n^*=0$};
\node[font=\footnotesize,anchor=north] at (0,-3.72) {(b) receding axis};
\end{tikzpicture}

\vspace{0.1cm}
\begin{tikzpicture}
\fill[blue!16] (0,0) rectangle (0.40,0.26);
\draw[red!65!black] (0,0) rectangle (0.40,0.26);
\node[font=\footnotesize,anchor=west] at (0.52,0.13) {$\Dcal_n^\flat$: all of the bounds~\eqref{eq:pointwise:zoom:bounds:a}--\eqref{eq:pointwise:zoom:bounds:e}, on compact sets and for $n$ large};
\fill[blue!6] (0,-0.42) rectangle (0.40,-0.16);
\draw[blue!60!black] (0,-0.42) rectangle (0.40,-0.16);
\node[font=\footnotesize,anchor=west] at (0.52,-0.29) {$\Dcal_n\setminus\Dcal_n^\flat$: the bounds~\eqref{eq:pointwise:zoom:bounds:a},~\eqref{eq:pointwise:zoom:bounds:b}, and~\eqref{eq:pointwise:zoom:bounds:d}};
\end{tikzpicture}

\vspace{-0.1cm}
\caption{\footnotesize The domains $\Dcal_n$ and $\Dcal_n^\flat$ are centered at $-a_n/\ell_n$, with radii $1/\ell_n$ and $3/(4\ell_n)$. The dashed line is the rescaled axis. In (a), $y=0$ lies on the axis and the image $y_n^*$ of the selected point, defined in Section~\ref{sec:pointwise:applications}, remains at bounded distance from it. In (b), $y_n^*=0$ and the axis recedes. The dotted ball $B(\rho_n)$ is centered at $y=0$, and its radius $\rho_n=1/(4\ell_n)$ tends to infinity.}
\label{fig:pointwise:domains}
\end{figure}

With the definitions in~\eqref{eq:pointwise:fixed:zoom}, we return to~\eqref{eq:meridional:nse} and obtain that 
\begin{subequations}
\label{eq:pointwise:meridional:equation}
\begin{align}
\p_\tau B_n+\div(B_n\otimes B_n)+\nabla P_n
&=\eps_n\Delta B_n-\div(Z_n\otimes Z_n)
\,,
\\
\div B_n &=0
\,,
\end{align}
\end{subequations}
holds on $\Dcal_n$.
We set 
\[
L_n(\tau)= \frac{\ell(t_n|\tau|)}{\ell_n} = \frac{\ell(t_n|\tau|)}{\ell(t_n)} \,,
\]
so that $L_n(-1)=1$. For $\tau\le-1$, the monotonicity of $\ell$ gives $L_n(\tau)\ge1$, and the doubling condition in~\eqref{eq:pointwise:length:doubling} gives
\begin{align}
1\leq L_n(\tau)\leq C_T:=D^{\lceil\log_2T\rceil}
\,,
\qquad
\mbox{for all}
\qquad \tau \in [-T,-1]
\,,
\label{eq:pointwise:length:ratios}
\end{align}
once $2^{\lceil\log_2T\rceil}(-t_n)<1$, which is what the iteration requires.

We now change variables in~\eqref{eq:pointwise:velocity},~\eqref{eq:pointwise:circulation},~\eqref{eq:pointwise:meridional:bounded}, and in~\eqref{eq:pointwise:quotient}, and we obtain, for every compact set of $(y,\tau)$ with $\tau\le-1$ and all $n$ large depending on it,
\begin{subequations}
\label{eq:pointwise:zoom:bounds}
\begin{align}
\max\big\{|B_n(y,\tau)|,|Z_n(y,\tau)|\big\}&\leq \frac{C_0L_n(\tau)}{|\tau|}\,,
  &&y\in\Dcal_n\,,
\label{eq:pointwise:zoom:bounds:a}
\\
|\nabla^2B_n(y,\tau)|&\le\frac{C_2}{|\tau|L_n(\tau)}\le\frac{C_2}{|\tau|}\,,
  &&y\in\Dcal_n\,,
\label{eq:pointwise:zoom:bounds:b}
\\
|\nabla B_n(y,\tau)|&\le\frac{C_1}{|\tau|}\,,
  &&y\in\Dcal_n^\flat\,,
\label{eq:pointwise:zoom:bounds:c}
\\
|\Omega_n(y,\tau)|&\le\frac{\sqrt2C_2}{|\tau|L_n(\tau)}\le\frac{\sqrt2C_2}{|\tau|}\,,
  &&y\in\Dcal_n\,,
\label{eq:pointwise:zoom:bounds:d}
\\
R_n(y)\,|Z_n(y,\tau)|&\leq C_\Gamma\eps_n\,,
  &&y\in\Dcal_n^\flat\,,
\label{eq:pointwise:zoom:bounds:e}
\end{align}
\end{subequations}
where 
\[
R_n(y):= \frac{|a_n'+\ell_ny'|}{\ell_n}
\] 
is the rescaled distance to the original axis, and we have defined 
\begin{align}
\Omega_n(y,\tau)&:=\frac{\Theta_n(y,\tau)}{R_n(y)}=(-t_n)\ell_n\,\Omega\big(a_n+\ell_ny,t_n|\tau|\big)\,.
\end{align}
Since $R_n(y)|Z_n(y,\tau)|=\eps_n|\Gamma(a_n+\ell_ny,t_n|\tau|)|$, 
the bound~\eqref{eq:pointwise:zoom:bounds:e} is precisely estimate~\eqref{eq:pointwise:circulation}. Here Lemma~\ref{lem:pointwise:local:circulation} is applied with $R_1=3/4$, some $R_0\in(3/4,1)$, $R_*\in(R_-,R_+)$, and $t_0<0$, so that $a_n+\ell_ny\in B(3/4)\subset B(R_*)$ for all $y\in\Dcal_n^\flat$; the time $t_n|\tau|$ lies in $(t_0,0)$ for all $n$ large depending on $\tau$, because $t_n\to0$.

We note that the base points $a_n$ occur in two configurations, according to their position relative to the axis of symmetry: 
\begin{itemize}[leftmargin=2em]
\item In the \emph{finite-axis} configuration $a_n$ lies on the axis, so that $R_n(y)=|y'|$. The quotient $\Omega_n=(\curl B_n)_\theta/|y'|$ is the potential vorticity of $B_n$ with respect to the rescaled axis $\{y'=0\}$. 
\item In the \emph{receding-axis} configuration $d_n:=r_n/\ell_n\to\infty$, where $r_n:=|a_n'|$. Then $R_n(y)\ge d_n-|y'|$, and we do not use the quotient $\Omega_n$. After a rotation about the axis, we identify $a_n'=(r_n,0)$.
\end{itemize}

\subsection{Two ancient-limit lemmas}
\label{sec:pointwise:limits}
Lemmas~\ref{lem:ancient:axis} and~\ref{lem:ancient:receding} below are at the heart of the limit argument. They differ in whether the axis is present in the limit, and in which scalar we transport; everything else is common to them. 

\begin{lemma}[{\bf Ancient limits, the axis present}]
\label{lem:ancient:axis}
Let $t_n$, $a_n$, and the fields~\eqref{eq:pointwise:fixed:zoom} be as in Section~\ref{sec:pointwise:zoom}, with every $a_n$ lying on the axis of symmetry. Define $\Ocal:=\RR^3\times(-\infty,-1]$, and let $\bar L\colon(-\infty,-1]\to[1,\infty)$ be non-decreasing in $|\tau|$, bounded on bounded sets, and such that
\begin{align}
\limsup_{n\to\infty}\,\sup_{-T\le\tau\le-1}\frac{L_n(\tau)}{\bar L(\tau)}\le1
\qquad\mbox{for every }T>1
\,.
\label{eq:ancient:Lbar}
\end{align}
Then, there exists a continuous $\Theta\colon \Ocal \to \RR$, with $|\Theta(y,\tau)|\le\sqrt2\,C_1/|\tau|$ on $\Ocal$, and
\begin{align}
\Theta(\cdot,-1)=0\qquad\mbox{on }\RR^3
\,,
\label{eq:ancient:terminal}
\end{align}
such that after passing to a subsequence we have
\begin{align}
\Theta_n\longrightarrow\Theta
\qquad\mbox{uniformly on compact subsets of }\Ocal
\,.
\label{eq:pointwise:Theta:convergence}
\end{align}
Moreover, along a further subsequence, $B_n(\cdot,-1)\to B^*$ in $C^1$ on compact subsets of $\RR^3$, where $B^*$ is divergence free, curl free, and bounded by $C_0$, so that each of its components is harmonic.
\end{lemma}

\begin{lemma}[{\bf Ancient limits, the axis receding}]
\label{lem:ancient:receding}
Let $t_n$, $a_n$, and the fields~\eqref{eq:pointwise:fixed:zoom} be as in Section~\ref{sec:pointwise:zoom}, with $a_n'=(r_n,0)$ and $d_n=r_n/\ell_n\to\infty$. Let $\Ocal$ and $\bar L$ be as in Lemma~\ref{lem:ancient:axis}. Then the conclusions of Lemma~\ref{lem:ancient:axis} all hold.
\end{lemma}

\begin{proof}[Proof of Lemma~\ref{lem:ancient:axis}]
For convenience, we break down the argument into five steps.

\emph{Step 1: compactness of the azimuthal vorticity.}
Since $\curl b=\omega_\theta e_\theta$, by~\eqref{eq:pointwise:fixed:zoom}
we have that 
\[
\curl B_n(y,\tau) =\Theta_n(y,\tau) \,e_\theta(y)
\] 
for all $y$ off the axis. Here we have used that $e_\theta(a_n+\ell_ny)=e_\theta(y)$ because $a_n$ lies on the axis. For $y$ on the axis, both $\Theta_n$ and $\curl B_n$ vanish, and we deduce that 
\[
|\curl B_n|=|\Theta_n|
\]
holds at every point of $\mathcal D_n$.  Since $|\curl B_n|\le\sqrt2\,|\nabla B_n|$, the bound~\eqref{eq:pointwise:zoom:bounds:c} gives $|\curl B_n|\le\sqrt2C_1/|\tau|\le\sqrt2C_1$ on every compact subset of $\Ocal$, for all $n$ large depending on the set. The bound~\eqref{eq:pointwise:zoom:bounds:b} makes $\curl B_n$ Lipschitz in $y$, with constant $\sqrt2C_2$, on every closed ball contained in $\Dcal_n$. By Lemma~\ref{lem:pointwise}(b) applied to $u(\cdot,t)$ on the convex axisymmetric ball $B(1)$, and using the change of variables in~\eqref{eq:pointwise:fixed:zoom}, the scalars $\Theta_n(\cdot,\tau)$ are Lipschitz on $\Dcal_n$ with the same constant $\sqrt2C_2$. We use both Lipschitz bounds only on the closed balls appearing in the next paragraph.

Taking the curl of~\eqref{eq:pointwise:meridional:equation}, we obtain for every $\bar y\in\RR^3$, $\rho>0$,  $\tau_a<\tau_b\le-1$,  $1<p<\infty$, and all $n$ large depending on these,
\begin{align*}
\norm{\p_\tau\curl B_n}_{L^\infty(\tau_a,\tau_b;W^{-2,p}(B(\bar y,\rho)))}\leq C_{\tau_a,\rho,p}
\,.
\end{align*}
Here $W^{-2,p}(B(\bar y,\rho))$ denotes the dual of $W^{2,p'}_0(B(\bar y,\rho))$, and we used the fact that  the quadratic terms are two derivatives of tensor fields bounded by~\eqref{eq:pointwise:zoom:bounds:a}, the viscous term is $\eps_n\Delta\curl B_n$ with $\curl B_n$ and $\eps_n$ bounded, and the pressure term is gone.  
Since the embedding $C^{0,1}(\overline{B(\bar y,\rho)})\hookrightarrow C^0(\overline{B(\bar y,\rho)})$ is compact and $C^0(\overline{B(\bar y,\rho)})$ embeds continuously and injectively into $W^{-2,p}(B(\bar y,\rho))$, a standard compactness argument gives, for every $\eta>0$, a constant $C_\eta$ with
\[
\norm{F}_{C^0(\overline{B(\bar y,\rho)})}\le\eta\norm{F}_{C^{0,1}(\overline{B(\bar y,\rho)})}+C_\eta\norm{F}_{W^{-2,p}(B(\bar y,\rho))}
\qquad\mbox{for all }F\in C^{0,1}(\overline{B(\bar y,\rho)})
\,.
\]
Applying this bound to $F=\curl B_n(\tau)-\curl B_n(\tau')$, using~\eqref{eq:pointwise:zoom:bounds:b},~\eqref{eq:pointwise:zoom:bounds:c}, and the previous estimate for $\partial_\tau \curl B_n$, we find that the maps $\tau\mapsto\curl B_n(\tau)$ are equicontinuous with values in $C^0(\overline{B(\bar y,\rho)})$. The unit vector $e_\theta(a_n+\ell_ny)$ does not depend on $\tau$, so that the scalars $\Theta_n$ are equicontinuous in $\tau$ as well, and they are equicontinuous in $y$ by the Lipschitz bound  recorded in the previous paragraph. Every compact subset of $\Ocal$ lies in one such product, being bounded in $y$ and in $\tau$, and $\Ocal$ is the union of the products $\overline{B(0,k)}\times[-k,-1]$ with $k\ge1$; applying the Arzel\`a-Ascoli theorem on each of these and extracting diagonally, we obtain~\eqref{eq:pointwise:Theta:convergence}, where $\Theta$ is continuous on $\Ocal$. The bound $|\Theta(y,\tau)|\le\sqrt2\,C_1/|\tau|$ is inherited from~\eqref{eq:pointwise:zoom:bounds:c}. 

The quotients $\Omega_n=\Theta_n/|y'|$ converge uniformly on compact subsets of $\Ocal\cap\{y'\neq0\}$ to $\Omega_\infty:=\Theta/|y'|$, and by~\eqref{eq:pointwise:zoom:bounds:d} we have 
\[
|\Omega_\infty(y,\tau)|\le\frac{\sqrt2C_2}{|\tau|} \, .
\]

\emph{Step 2: the limiting meridional velocity.}
By~\eqref{eq:pointwise:zoom:bounds:a} and~\eqref{eq:pointwise:zoom:bounds:c}, the fields $B_n$ and $\nabla B_n$ are bounded on every compact subset of $\Ocal$. After a further diagonal extraction we have $B_n\rightharpoonup B$ and $\nabla B_n\rightharpoonup\nabla B$ weakly-$*$ in $L^\infty$ on every such set, with $\nabla B$ the distributional spatial gradient of $B$. The limit satisfies $\div B=0$, $\curl B=\Theta\,e_\theta$, and due to~\eqref{eq:pointwise:zoom:bounds:c},~\eqref{eq:pointwise:zoom:bounds:a}, and~\eqref{eq:ancient:Lbar} it obeys
\begin{align}
|\nabla B(y,\tau)|\le\frac{C_1}{|\tau|}\,,
\qquad
|B(y,\tau)|\le\frac{C_0\,\bar L(\tau)}{|\tau|}
\qquad\mbox{for almost every }(y,\tau)\in\Ocal
\,.
\label{eq:pointwise:limit:velocity}
\end{align}
Since $B$ is defined only up to a null set, for almost every $\tau$ we replace $B(\cdot,\tau)$ by its continuous representative, which is divergence free, Lipschitz on $\RR^3$ with constant $C_1/|\tau|$, and bounded by $C_0\bar L(\tau)/|\tau|$ at every $y\in\RR^3$. On the exceptional null set of times we set $B(\cdot,\tau)=0$. The resulting field is continuous in $y$ at every $\tau$ by construction, and measurable in $\tau$ at every $y$. The Carath\'eodory solutions of $\dot X(\tau)=B(X(\tau),\tau)$, that is, the absolutely continuous $X$ which satisfy the equation in integral form, do not see the exceptional set. The fields $B_n$ are axisymmetric and swirl free, and thus so is $B$. Since $B$ is continuous in $y$, its radial component $B_r$ vanishes on the axis, and the Lipschitz bound gives $|B_r(y,\tau)|\leq C_1|y'|/|\tau|$.

\emph{Step 3: the transport equation for the potential vorticity.}
From~\eqref{eq:Omega:eqn}, the change of variables~\eqref{eq:pointwise:fixed:zoom}, and $\div B_n=0$, we obtain that the potential vorticity $\Omega_n$ satisfies
\begin{align}
\p_\tau\Omega_n+ \div \Bigl( B_n \, \Omega_n\Bigr)
=\eps_n\Big(\Delta+\frac{2}{|y'|}\p_{|y'|}\Big)\Omega_n+\p_z\Big(\frac{|Z_n|^2}{|y'|^2}\Big)
\qquad\mbox{on }\{y\in\Dcal_n,\ y'\neq0\}
\,.
\label{eq:pointwise:Omega:n:eqn}
\end{align}
We fix $\varphi\in C^\infty_0$ compactly supported in the interior of $\Ocal\cap\{y'\neq0\}$ and test~\eqref{eq:pointwise:Omega:n:eqn} against $\varphi$. For all large enough $n$, we have $\supp(\varphi) \subset \Dcal_n^\flat$. The viscous term contributes at most $\eps_n\norm{\Omega_n}_{L^\infty}\int(|\Delta\varphi|+2|y'|^{-1}|\p_{|y'|}\varphi|)$, which tends to zero as $n\to \infty$ in light of~\eqref{eq:pointwise:zoom:bounds:d}. The swirl source contributes $-\int|y'|^{-2}|Z_n|^2\p_z\varphi$, which tends to zero as $n\to \infty$ because $|Z_n|\leq C_\Gamma\eps_n/|y'|$ on the support of $\varphi$ by~\eqref{eq:pointwise:zoom:bounds:e}. For the drift term,  we have $\Omega_n\nabla\varphi\to\Omega_\infty\nabla\varphi$ in $L^1$ by \emph{Step~1}, and $B_n\rightharpoonup B$ weakly-$*$ in $L^\infty$ by \emph{Step~2}. Passing $n\to \infty$ we therefore obtain
\begin{align}
\int_{\Ocal\cap\{y'\ne0\}}\Omega_\infty\big(\p_\tau\varphi+B\cdot\nabla\varphi\big)\,dy\,d\tau=0
\,,
\label{eq:pointwise:ancient:transport}
\end{align}
for every $\varphi\in C^\infty_0$ compactly supported in the interior of $\Ocal\cap\{y'\neq0\}$. 
In~\eqref{eq:pointwise:ancient:transport} the transported function $\Omega_\infty$ is continuous and bounded on its domain, and the velocity is bounded on compact subsets of $\Ocal$, divergence free, measurable in time, and Lipschitz in space on compact sets.

\emph{Step 4: characteristics and the vanishing of the limit vorticity.}
We fix a point $y_0\in\RR^3$ off the axis. Since $B(\cdot,\tau)$ is Lipschitz on $\RR^3$ with constant $C_1/|\tau|$ and is bounded by $C_0\bar L(\tau)/|\tau|$, the equation $\dot X(\tau)=B(X(\tau),\tau)$ with $X(-1)=y_0$ has a unique solution $X$ on $(-\infty,-1]$, absolutely continuous on compact subintervals. In order to apply Lemma~\ref{lem:pointwise:characteristics} on $\Ucal:=\Ocal\cap\{y'\ne0\}$, we need the graph of $X$ over $[\tau_0,-1]$ to be a compact subset of $\Ucal$ for every $\tau_0<-1$. This follows from two a priori bounds. First, the velocity bound in~\eqref{eq:pointwise:limit:velocity} gives
\begin{align}
|X(\tau)|\le|y_0|+C_0\int_\tau^{-1}\frac{\bar L(\bar \tau)}{|\bar \tau|}\,d\bar\tau
\,,
\label{eq:ancient:displacement}
\end{align}
whose right side is finite on $[\tau_0,-1]$, since $\bar L$ is bounded on bounded sets. The path's distance to the axis, denoted by $\varrho(\tau):=|X'(\tau)|$, is absolutely continuous. Since $B$ is swirl free with $|B'(y,\tau)|=|B_r(y,\tau)|\le C_1|y'|/|\tau|$ by \emph{Step~2}, it follows that $|\dot\varrho|\le|B'(X,\tau)|\le C_1\varrho/|\tau|$ almost everywhere. Gr\"onwall's inequality on $[\tau,-1]$ then yields $\varrho(-1)\le\varrho(\tau)\,|\tau|^{C_1}$, that is,
\begin{align*}
\varrho(\tau)\ge|y_0'|\,|\tau|^{-C_1} \,, \qquad \mbox{for} \qquad \tau\le-1\,.
\end{align*}
Hence, the graph of $X$ over $[\tau_0,-1]$ lies in $\{(y,\tau)\colon\tau_0\le\tau\le-1,\ |y|\le M(\tau_0),\ |y'|\ge|y_0'|\,|\tau_0|^{-C_1}\}$, where $M(\tau_0)$ is the right side of~\eqref{eq:ancient:displacement} evaluated at $\tau=\tau_0$; this is a compact subset of $\Ucal$.

By~\eqref{eq:pointwise:ancient:transport} and \emph{Step~2}, the pair $(\Omega_\infty,B)$ on $\Ucal = \Ocal\cap\{y'\ne0\}$ satisfies the hypotheses of Lemma~\ref{lem:pointwise:characteristics}, with $\lambda(\tau)=C_1/|\tau|$. From the last estimate in \emph{Step~1}, we therefore have 
\[
|\Omega_\infty(y_0,-1)|=\lim_{\tau\to-\infty}|\Omega_\infty(X(\tau),\tau)|\le\lim_{\tau\to-\infty}  \frac{\sqrt2C_2}{|\tau|} =0 \,,
\] 
which implies $\Omega_\infty(\cdot,-1)=0$ on $\{ y \in \RR^3 \colon |y'|>0\}$. This proves~\eqref{eq:ancient:terminal}, since $\Theta(\cdot,-1)=|y'|\Omega_\infty(\cdot,-1)$ off the axis and $\Theta$ is continuous.

\emph{Step 5: extracting the limit at the terminal time.}
Being a weak-$*$ limit in $L^\infty$ of the sequence $B_n$ (see \emph{Step~2}), the field $B$ is determined only up to null sets of $\Ocal$, and the slice $\{\tau=-1\}$ is one. In order to obtain a limit for the fields $B_n(\cdot,-1)$, we appeal to a compactness argument rooted in the convergence of the curl of these fields to $\Theta e_\theta$ (see \emph{Step~1}). By~\eqref{eq:pointwise:zoom:bounds:a},~\eqref{eq:pointwise:zoom:bounds:b}, and~\eqref{eq:pointwise:zoom:bounds:c} at $\tau=-1$, where $L_n(-1)=1$, the fields $B_n(\cdot,-1)$ are bounded in $C^{1,1}$ on every compact subset of $\RR^3$ and satisfy $|B_n(\cdot,-1)|\leq C_0$ there. Passing to a further subsequence, we obtain $B_n(\cdot,-1)\to B^*$ in $C^1$ on compact subsets of $\RR^3$. It automatically holds that $\div B^*=0$ and $|B^*|\leq C_0$. This subsequence refines the one of \emph{Step~1}, and $|\curl B_n|=|\Theta_n|$ pointwise, so that $|\curl B^*|=|\Theta(\cdot,-1)|$. Due to~\eqref{eq:ancient:terminal} we thus obtain also $\curl B^*=0$. By elliptic regularity $B^*$ is smooth and each of its components is harmonic.
\end{proof}

\begin{proof}[Proof of Lemma~\ref{lem:ancient:receding}]
We follow the five steps of the proof of Lemma~\ref{lem:ancient:axis}, and we only record the necessary changes.

\emph{Step 1: compactness of the azimuthal vorticity.}
Since $\curl b=\omega_\theta e_\theta$, by~\eqref{eq:pointwise:fixed:zoom} we now have
\[
\curl B_n(y,\tau)=\Theta_n(y,\tau)\,e_\theta(a_n+\ell_ny)
\]
wherever $R_n(y) = |a_n' + \ell_n y'|/\ell_n>0$. Since $a_n$ does not lie on the symmetry axis we cannot equate $e_\theta(a_n+\ell_ny)$ with $e_\theta(y)$; however, because $d_n\to\infty$ and $a_n'=(r_n,0)$ we have $e_\theta(a_n+\ell_ny)\to e_2$ uniformly on compact subsets of $\RR^3$. Since $R_n(y)\ge r_n/\ell_n - |y'| \geq d_n-|y|$, we deduce that $|\curl B_n|=|\Theta_n|$ on every compact subset of $\RR^3$, for all $n$ large depending on the set. The bounds on $\curl B_n$ and on $\Theta_n$, the negative-norm estimate for $\p_\tau\curl B_n$, and the compactness argument of \emph{Step~1} of the proof of Lemma~\ref{lem:ancient:axis} apply verbatim, and give~\eqref{eq:pointwise:Theta:convergence}, where $\Theta$ is continuous on $\Ocal$ with $|\Theta(y,\tau)|\le\sqrt2\,C_1/|\tau|$. Here we do not consider the limit of the potential vorticities $\Omega_n$.

\emph{Step 2: the limiting meridional velocity.}
Exactly as in \emph{Step~2} of the proof of Lemma~\ref{lem:ancient:axis}, after a further diagonal extraction we have $B_n\rightharpoonup B$ and $\nabla B_n\rightharpoonup\nabla B$ weakly-$*$ in $L^\infty$ on every compact subset of $\Ocal$, the limit satisfies $\div B=0$, $\curl B=\Theta\,e_2$, and the bounds in~\eqref{eq:pointwise:limit:velocity}. We fix a representative in the same way. 

\emph{Step 3: the transport equation for the azimuthal vorticity.}
In this proof we transport $\Theta_n$ itself. From~\eqref{eq:vort:theta} and the change of variables~\eqref{eq:pointwise:fixed:zoom} we obtain
\begin{align}
\p_\tau\Theta_n+B_n\cdot\nabla\Theta_n
=\frac{(B_n)_r}{R_n}\Theta_n+\frac{1}{R_n}\p_z\big(|Z_n|^2\big)+\eps_n\Big(\Delta-\frac1{R_n^2}\Big)\Theta_n
\qquad\mbox{on }\{y\in\Dcal_n\colon R_n(y)>0\}
\,,
\label{eq:pointwise:Theta:n:eqn}
\end{align}
where $(B_n)_r$ is the radial component with respect to the original axis. We fix $\varphi\in C^\infty_0$ compactly supported in the interior of $\Ocal$ and test~\eqref{eq:pointwise:Theta:n:eqn} against $\varphi$. On $\supp\varphi$ we have $R_n(y)\ge d_n-|y|\to\infty$ and $|\p_z(\varphi/R_n)|\leq C/(d_n-|y|)$, so that every term on the right side tends to zero as $n\to\infty$, by the vanishing $\eps_n \to 0$ and the bounds~\eqref{eq:pointwise:zoom:bounds:a} and~\eqref{eq:pointwise:zoom:bounds:c} on $B_n$, $Z_n$, and $\Theta_n$. Here the receding axis, namely the fact that $R_n(y)\to\infty$, plays the same role as the circulation bound in \emph{Step~3} of the proof of Lemma~\ref{lem:ancient:axis}: it removes the stretching term and the swirl source, while the viscous terms vanish because $\eps_n\to0$. Passing $n\to\infty$ in the drift term as in that step, we obtain
\begin{align}
\int_{\Ocal}\Theta\big(\p_\tau\varphi+B\cdot\nabla\varphi\big)\,dy\,d\tau=0
\,,
\label{eq:pointwise:ancient:transport:planar}
\end{align}
for every $\varphi\in C^\infty_0$ compactly supported in the interior of $\Ocal$. Here too the transported function is continuous and bounded on its domain, and the velocity is bounded on compact subsets of $\Ocal$, divergence free, measurable in time, and Lipschitz in space on compact sets.

\emph{Step 4: characteristics and the vanishing of the limit vorticity.}
We fix a point $y_0\in\RR^3$. As in \emph{Step~4} of the proof of Lemma~\ref{lem:ancient:axis}, the equation $\dot X(\tau)=B(X(\tau),\tau)$ with $X(-1)=y_0$ has a unique solution $X$ on $(-\infty,-1]$, absolutely continuous on compact subintervals, and the displacement bound~\eqref{eq:ancient:displacement} holds. Since here $\Ucal=\Ocal$, that bound alone confines the graph of $X$ over $[\tau_0,-1]$ to a compact subset of $\Ucal$, and we need no lower bound on the distance to the axis. By~\eqref{eq:pointwise:ancient:transport:planar} and \emph{Step~2}, the pair $(\Theta,B)$ satisfies the hypotheses of Lemma~\ref{lem:pointwise:characteristics}, with $\lambda(\tau)=C_1/|\tau|$, so that
\[
|\Theta(y_0,-1)|=\lim_{\tau\to-\infty}|\Theta(X(\tau),\tau)|\le\lim_{\tau\to-\infty}\frac{\sqrt2C_1}{|\tau|}=0
\,,
\]
which proves~\eqref{eq:ancient:terminal}.

\emph{Step 5: extracting the limit at the terminal time.}
As in \emph{Step~5} of the proof of Lemma~\ref{lem:ancient:axis}, the weak-$*$ convergence to the field $B$ gives no limit for the fields $B_n(\cdot,-1)$. By~\eqref{eq:pointwise:zoom:bounds:a},~\eqref{eq:pointwise:zoom:bounds:b}, and~\eqref{eq:pointwise:zoom:bounds:c} at $\tau=-1$, where $L_n(-1)=1$, the fields $B_n(\cdot,-1)$ are bounded in $C^{1,1}$ on every compact subset of $\RR^3$ and satisfy $|B_n(\cdot,-1)|\leq C_0$ there. Passing to a further subsequence, we obtain $B_n(\cdot,-1)\to B^*$ in $C^1$ on compact subsets of $\RR^3$, with $\div B^*=0$ and $|B^*|\leq C_0$. This subsequence refines the one of \emph{Step~1}. Since $|\curl B_n|=|\Theta_n|$ on compact sets for $n$ large, we obtain $|\curl B^*|=|\Theta(\cdot,-1)|$, which vanishes by~\eqref{eq:ancient:terminal}. By elliptic regularity $B^*$ is smooth and each of its components is harmonic.
\end{proof}

\subsection{Proof of Proposition~\ref{prop:pointwise:expanding}}
\label{sec:pointwise:applications}
Since $\ell(t)\to0$, we may fix $T_1\in(0,1]$ for which~\eqref{eq:pointwise:T:one} holds, so that the fields and the bounds of Section~\ref{sec:pointwise:zoom} are available. In particular,~\eqref{eq:pointwise:meridional:bounded} shows that $\sup_{t\in(-T_1,0)}(-t)\norm{\nabla b(\cdot,t)}_{L^\infty(B(1/2))}\leq C_1$.

Suppose by contradiction that~\eqref{eq:pointwise:expanding:conclusion} fails. Then there exist $c_0>0$, times $t_n\uparrow0$ with $t_n\in(-T_1,0)$, and points $x_n\in B(1/2)$ with
\begin{align}
(-t_n)|\nabla b(x_n,t_n)|\ge c_0
\,.
\label{eq:pointwise:selected}
\end{align}
By rotational invariance we may without loss of generality take $x_n'=(r_n,0)$. Recall that $\ell_n = \ell(t_n)$. After passing to a subsequence, either $r_n/\ell_n$ is bounded (finite-axis), or $r_n/\ell_n\to\infty$ (receding-axis). These are the two configurations discussed in Section~\ref{sec:pointwise:zoom}, and we choose the points $a_n$ from~\eqref{eq:pointwise:fixed:zoom} accordingly: in the finite-axis configuration we define $a_n:=(0,0,x_n\cdot e_z)$, while in the receding-axis configuration we define $a_n:=x_n$. In both cases $|a_n|\le|x_n|<\frac12$, and the image $y_n^*:=(x_n-a_n)/\ell_n$ of $x_n$ under the zoom-in from~\eqref{eq:pointwise:fixed:zoom} is $(r_n/\ell_n,0,0)$ and $(0,0,0)$ respectively, hence bounded.

In terms of the doubling constant $D\geq 1$ from~\eqref{eq:pointwise:length:doubling}, we define $\bar L\colon(-\infty,-1]\to[1,\infty)$ by
\begin{align*}
\bar L(\tau):=D^{\lceil\log_2|\tau|\rceil}
\,.
\end{align*}
This function is non-decreasing in $|\tau|$ and satisfies $\bar L(\tau)\leq C_T$ for $-T\le\tau\le-1$, with $C_T$ as in~\eqref{eq:pointwise:length:ratios}. By monotonicity and doubling, $L_n(\tau) \le D^k$ for all $\tau\in[-2^k,-2^{k-1})\cap[-T,-1]$ and $0\le k\le\lceil\log_2T\rceil$, for all large $n$ (depending on $T$). Therefore, $L_n\le\bar L$ on $[-T,-1]$ for every $T>1$, and all $n$ large enough depending on $T$. Thus, we have verified assumption~\eqref{eq:ancient:Lbar}.

We apply Lemma~\ref{lem:ancient:axis} in the finite-axis configuration and Lemma~\ref{lem:ancient:receding} in the receding-axis configuration. In both cases, along a subsequence $B_n(\cdot,-1)\to B^*$ in $C^1$ on compact subsets of $\RR^3$, with each component of $B^*$ being harmonic. Moreover, $|B^*|\leq C_0$. A bounded harmonic function on $\RR^3$ is constant, and thus $\nabla B^*\equiv0$.

Since $y_n^*$ is bounded, after passing to a further subsequence we have $y_n^*\to y^*$. Differentiating~\eqref{eq:pointwise:fixed:zoom}, we obtain $\nabla B_n(y_n^*,-1)=(-t_n)\nabla b(x_n,t_n)$, and since $\nabla B_n(\cdot,-1)\to\nabla B^*$ uniformly on compact sets, passing to the limit in~\eqref{eq:pointwise:selected} yields $|\nabla B^*(y^*)|\ge c_0$. This contradicts $\nabla B^*\equiv0$, and Proposition~\ref{prop:pointwise:expanding} is proved.
\qed

\section{The local regularity conclusion}
\label{sec:pointwise:regularity}
In this section we use the asymptotic vanishing statement~\eqref{eq:pointwise:meridional:small} (which for $\rho_0=1/2$ is~\eqref{eq:pointwise:expanding:conclusion}) to prove  regularity of the flow on a smaller fixed ball, which completes the proof of Theorem~\ref{thm:pointwise:variable}.  

\begin{proposition}[{\bf Smallness of $(-t)\norm{\nabla b(\cdot,t)}_{L^\infty}$ implies regularity}]
\label{prop:pointwise:meridional:small}
Suppose that~\eqref{eq:pointwise:meridional:small} holds for some $\rho_0\in(0,1)$. Then $(0,0)$ is a regular point.
\end{proposition}

\begin{proof}[Proof of Proposition~\ref{prop:pointwise:meridional:small}]
We fix a regular annulus for $u$ inside $B(\rho_0)$, as in
Lemma~\ref{lem:pointwise:local:circulation}; we also fix the time $t_0$ and the radii $R_-<R_+$ this lemma supplies, together with an $R_*\in(R_-,R_+)$. Let $\chi$ be a smooth radial cutoff
supported in $B(R_*)$, equal to one on $B(\rho)$ for some $\rho\in(R_-,R_*)$, so that the
derivatives of $\chi$ are supported in the regular annulus $\{R_-<|x|<R_+\}$ of $u$.  

We write $G(t)=\norm{\nabla b(\cdot,t)}_{L^\infty(B(\rho_0))}$, and for each
$\eta>0$ we fix $t_\eta\in(t_0,0)$ so close to zero that
\begin{equation}
\label{eq:pointwise:G:bound}
G(t)\le \frac{\eta}{(-t)}\, \qquad \mbox{on} \qquad (t_\eta,0)
\,.
\end{equation}
Note that since the radial component $b_r=u_r$ of $b$ vanishes on the axis we have
$|b_r/r|\leq G$. Moreover, the azimuthal velocity $u_\theta$ vanishes on the axis of symmetry, is bounded on $\partial B(R_*)\times(t_0,0)$ and on the initial time-slice $B(R_*) \times \{t_\eta\}$. From the swirl equation~\eqref{eq:swirl:eqn} we have that the function $\psi=(-t)^\eta u_\theta$
satisfies 
\[
\p_t\psi+b\cdot\nabla\psi-\Delta\psi
+\Bigl(r^{-2}+\frac{b_r}{r}+\frac{\eta}{(-t)}\Bigr)\psi=0
\,.
\]
Since $|b_r/r|\leq G(t)\le\eta/(-t)$, the zeroth order coefficient satisfies
\[
r^{-2}+\frac{b_r}{r}+\frac{\eta}{(-t)}\ge r^{-2}\ge0
\]
on $B(R_*)\times (t_\eta,0)$. Although the coefficient $r^{-2}$ is singular on the axis, $\psi$ vanishes there, since $u_\theta$ does, so that a positive maximum or a negative minimum of $\psi$ which is not attained on the parabolic boundary is attained off the axis, where the classical argument applies. Applying the maximum principle, we obtain a constant $C_\eta>0$ such that $|\psi| \leq C_\eta$ on $B(R_*)\times (t_\eta,0)$, and therefore
\begin{align}
W(t):=\norm{u_\theta(\cdot,t)}_{L^\infty(B(R_*))}=\norm{w(\cdot,t)}_{L^\infty(B(R_*))}
\leq C_\eta (-t)^{-\eta}
\,.
\label{eq:pointwise:swirl:slow}
\end{align}

We let
\[
Y(t) = \| \chi \, \omega(\cdot,t)\|_{L^2}^2\,,
\qquad
M(t)= \| \chi \, \nabla \omega(\cdot,t)\|_{L^2}^2\,,
\]
where as usual, $\omega=\curl u$. Taking the curl of~\eqref{eq:nse}, and splitting the advecting velocity as $u=b+w$ (cf.~\eqref{eq:axi:b:w:def}), we obtain
\[
\p_t\omega-\Delta\omega+\curl\bigl((b\cdot\nabla)u\bigr)+\curl\bigl((w\cdot\nabla)u\bigr)=0
\,.
\]
We test this equation against $\chi^2\omega$. Because $u$ and all its spatial derivatives are bounded on the regular annulus, and thus on $\supp(\nabla\chi)$, every term in which a derivative falls on $\chi$ is bounded by a constant depending only on $\chi$ and the $u$-bounds. The diffusion term therefore yields $M$ up to such an a priori bounded term. For the nonlinear contribution from the meridional term, since $\div b=0$  we have the identity
\[
\bigl(\curl\bigl((b\cdot\nabla)u\bigr)\bigr)_i
=\p_j(b_j\omega_i)+\varepsilon_{ik\ell}(\p_kb_j)(\p_ju_\ell)
\,,
\]
where $\varepsilon_{ik\ell}$ is the Levi-Civita symbol.
By the definitions of $G$ and $Y$, the second term is bounded by $G\norm{\chi\nabla u}_{L^2}Y^{1/2}$. The first term contributes only a priori bounded terms which contain $\nabla \chi$ because $\div b = 0$ implies 
$\int \p_j(b_j\omega_i) \omega_i \chi^2  = - \frac 12 \int b \cdot \nabla |\omega|^2 \chi^2  - 2 \int b\cdot\nabla \chi |\omega|^2 \chi = - \int b\cdot\nabla \chi |\omega|^2 \chi $.
For the contribution from the nonlinear swirl term, we move the $\curl$ onto $\chi^2\omega$, and by the definition of $W$ this contributes an a priori bounded term and a term bounded by  $2 W\norm{\chi\nabla u}_{L^2}M^{1/2}$. We have thus shown that there exists $C\geq 1$, which depends on the bounds of $u$ and its derivatives on the regular annulus, such that
\[
\frac 12 \frac{d}{dt} Y + M \leq G\norm{\chi\nabla u}_{L^2}Y^{1/2} + 2 W\norm{\chi\nabla u}_{L^2}M^{1/2} + C \,.
\]
It remains to bound $\norm{\chi\nabla u}_{L^2}$. The whole-space div-curl identity, applied to the compactly supported field $\chi u$, together with $\div u=0$, gives
\[
\norm{\nabla(\chi u)}_{L^2}^2\le C(Y+1)
\,,
\]
upon possibly enlarging the constant $C$; since $\nabla\chi$ is supported in the regular annulus, the same bound holds for $\norm{\chi\nabla u}_{L^2}^2$. 
Absorbing $M^{1/2}$ into the dissipation by Young's inequality, we arrive at 
\begin{align}
\frac{d}{dt} Y + M \leq C_*\bigl(G+W^2+1\bigr)(Y+1)
\,,
\label{eq:pointwise:enstrophy:split}
\end{align}
for a constant $C_*\ge1$. 

At this point we fix $\eta=1/(16C_*)\le1/16$; this choice is permissible because $C_*$ depends only on the cutoff $\chi$ and on the bounds on $u$ on the regular annulus, and these were fixed before $\eta$ was chosen. 
Combining~\eqref{eq:pointwise:enstrophy:split} with the bounds in~\eqref{eq:pointwise:G:bound} and~\eqref{eq:pointwise:swirl:slow}, with the above choice of $\eta$ we obtain from Gr\"onwall's inequality that
\[
Y(t) \leq (Y(t_\eta) + 1)
\exp\left( \int_{t_\eta}^t \Bigl(\frac{C_* \eta}{(-s)} + \frac{C_* C_\eta^2}{(-s)^{2\eta}} + C_* \Bigr) ds\right)
\leq C_\eta' (-t)^{-C_*\eta} 
\]
for all $t \in (t_\eta,0)$, for a suitable constant $C_\eta' >0$.
The div-curl estimate above and the homogeneous Sobolev inequality
therefore give
\[
\norm{u(\cdot,t)}_{L^6(B(\rho))}
\le\norm{\chi u(\cdot,t)}_{L^6(\RR^3)}
\leq C_\eta^{\prime\prime} (-t)^{-C_*\eta/2}
\,,
\]
for all $t \in (t_\eta,0)$. Here $B(\rho)$ is the fixed ball on which $\chi\equiv 1$. Since
$2C_*\eta=1/8<1$, the above bound implies a local Ladyzhenskaya-Prodi-Serrin regularity condition (cf.~Serrin~\cite{Serrin62} for $u\in L^q_tL^p_x$ with $3/p+2/q<1$, and Struwe~\cite[Theorem~3.1]{Struwe88} for $3/p+2/q=1$)
\[
u\in L^4(t_\eta,0;L^6(B(\rho))) \,.
\] 
The specific form of local regularity we need here may be found for instance in Gustafson, Kang, and Tsai~\cite[Corollary~1.2(i)]{GustafsonKangTsai07}: if $u \in L^4((-r^2,0);L^6(B(r)))$ for some $r>0$, then the top of the cylinder $(0,0)$ is a regular point in the sense of~\eqref{eq:intro:regular}. Here we apply this result with $r = \min\{ \rho, |t_\eta|^{1/2} \}$. 
\end{proof}

\begin{proof}[Proof of Theorem~\ref{thm:pointwise:variable}]
By Proposition~\ref{prop:pointwise:expanding}, the bounds~\eqref{eq:pointwise:velocity} give~\eqref{eq:pointwise:expanding:conclusion}, which is~\eqref{eq:pointwise:meridional:small} with $\rho_0=1/2$, and Proposition~\ref{prop:pointwise:meridional:small} makes $(0,0)$ a regular point.
\end{proof}

Note that Theorem~\ref{thm:pointwise:variable} gives regularity not just for $\ell(t) = (-t)^\gamma$ with $\gamma \in (0,1/2)$, but also for power laws with a logarithmic correction.

\begin{corollary}[{\bf Local power laws and logarithmic corrections}]
\label{cor:pointwise:logs}
Let $(u,\pi)$ be as in Theorem~\ref{thm:pointwise:variable}, with~\eqref{eq:pointwise:velocity} being replaced by
the two local bounds
\begin{align}
\norm{u(\cdot,t)}_{L^\infty(B(1))}
\leq C (-t)^{\gamma-1}\Big(\log\frac e{(-t)}\Big)^a \,,
\qquad
\norm{\nabla^2u(\cdot,t)}_{L^\infty(B(1))}
\leq C (-t)^{-1-\gamma}\Big(\log\frac e{(-t)}\Big)^{-a}
\,,
\label{eq:pointwise:log:bounds}
\end{align}
for all $t<0$ sufficiently close to zero, with 
\begin{itemize}[leftmargin=2em] 
\item either $0<\gamma<1/2$ and $a\in\RR$,  
\item or $\gamma=1/2$ and $a>0$.
\end{itemize}
Then $(0,0)$ is a regular point of the Navier-Stokes solution $(u,\pi)$.
\end{corollary}

\begin{proof}[Proof of Corollary~\ref{cor:pointwise:logs}]
We take $\ell(t)=(-t)^\gamma(\log(e/(-t)))^a$. Its logarithmic derivative is
\[
\frac{t\ell'(t)}{\ell(t)}
=\gamma-\frac{a}{\log(e/(-t))}
\,.
\]
Since $t<0$, in each stated range $\ell$ is non-increasing near zero, is
doubling, and tends to zero. We also have
\[
\frac{(-t)}{\ell(t)^2}
=(-t)^{1-2\gamma}\left(\log\frac e{(-t)}\right)^{-2a}\longrightarrow0
\,.
\]
Thus, for some $0<T_1\ll 1$, the function $\ell$ satisfies the requirements of Definition~\ref{def:Euler:length} with $(-T_1,0)$ and $(-T_1/2,0)$ in place of $(-1,0)$ and $(-1/2,0)$, and on $(-T_1,0)$ the bounds~\eqref{eq:pointwise:log:bounds} hold and are nothing but~\eqref{eq:pointwise:velocity}. Applying Theorem~\ref{thm:pointwise:variable} to $u_\lambda(x,t)=\lambda u(\lambda x,\lambda^2t)$, $\pi_\lambda(x,t)=\lambda^2\pi(\lambda x,\lambda^2t)$, with $\lambda=T_1^{1/2}$ and the Euler length $\lambda^{-1}\ell(\lambda^2t)$, for which~\eqref{eq:pointwise:velocity} holds with the same constants, we conclude.
\end{proof}

\section{Weaker vorticity assumptions}
\label{sec:sharp}
The second derivative bound in Theorem~\ref{thm:pointwise:variable} is more than the proof uses. In this section we state the form of the theorem which the proof actually gives, with a H\"older seminorm of the azimuthal vorticity and a bound on the potential vorticity in place of $\nabla^2u$. The only additional ingredient we need is an interior div-curl estimate.  

\begin{theorem}[{\bf Local regularity under matched vorticity bounds}]
\label{thm:sharp}
Let $0<\alpha \leq 1$. Let $(u,\pi)$ and $\ell$ be as in Theorem~\ref{thm:pointwise:variable}, with~\eqref{eq:pointwise:velocity} being replaced by the assumptions
\begin{align}
\norm{u(\cdot,t)}_{L^\infty(B(1))}
\leq C_0\frac{\ell(t)}{(-t)}
\,,
\quad
[\omega_\theta(\cdot,t)]_{C^\alpha(B(1))}
\le\frac{C_\alpha}{(-t)\ell(t)^\alpha}
\,,
\quad
\norm{\Omega(\cdot,t)}_{L^\infty(B(1))}
\le\frac{C_H}{(-t)\ell(t)}
\,,
\label{eq:sharp:matched}
\end{align}
for all $t\in(-1,0)$.\footnote{Here $\omega_\theta$ is viewed as a rotation-invariant scalar function on $B(1)$, equal to zero on the axis, and the H\"older seminorm is taken with respect to the Euclidean distance.}
Then $(0,0)$ is a regular point. When $\alpha=1$ the third bound follows from the second, and may be omitted.
\end{theorem}

For the power law $\ell(t)=(-t)^\gamma$, the three upper bounds in~\eqref{eq:sharp:matched} are  $(-t)^{\gamma-1}$, $(-t)^{-1-\alpha\gamma}$, and $(-t)^{-1-\gamma}$. That is, the rescaled field $U(y,t)=(-t)^{1-\gamma}u((-t)^\gamma y,t)$ is bounded, has $C^\alpha$ azimuthal vorticity, and has bounded potential vorticity on the expanding ball $B(1/(-t)^\gamma)$. Lemma~\ref{lem:pointwise}(b) turns the Hessian bound in~\eqref{eq:pointwise:velocity} into the second bound in~\eqref{eq:sharp:matched} at $\alpha=1$, and Lemma~\ref{lem:pointwise}(a) turns it into the third. Theorem~\ref{thm:pointwise:variable} is therefore the case $\alpha=1$ of Theorem~\ref{thm:sharp}.

\subsection{Interior div-curl estimates}
\label{sec:pointwise:interpolation}

The bounds in~\eqref{eq:sharp:matched} control the velocity and the H\"older seminorm of the azimuthal vorticity. The following lemma converts them into a gradient bound, and after the zoom-in into a uniform $C^{1,\beta}$ bound. Part~(a) passes from the scalar $\omega_\theta$ to the vector $\omega_\theta e_\theta$, and part~(b) records the interior Schauder estimates for the div-curl system, as used here.

\begin{lemma}[{\bf Interior div-curl estimates}]
\label{lem:pointwise:divcurl}
Let $0<\alpha\le1$.
\begin{enumerate}[label=(\alph*),leftmargin=2em]
\item Let $\varpi\in C^\alpha(B(R))$ be invariant under rotations about the $z$-axis and equal to zero on it. Then, with $\varpi\,e_\theta$ extended by zero to the axis, we have 
\begin{align}
[\varpi\,e_\theta]_{C^\alpha(B(R))}\le3\,[\varpi]_{C^\alpha(B(R))}
\,.
\label{eq:pointwise:vector:holder}
\end{align}
\item Let $F$ be a smooth divergence-free field on a ball $B(x_0,R)$, and let $S=\norm{F}_{L^\infty(B(x_0,R))}$ and $K=[\curl F]_{C^\alpha(B(x_0,R))}$. Then
\begin{align}
\norm{\nabla F}_{L^\infty(B(x_0,R/2))}
\leq C(\alpha)\Big( S R^{-1}+S^{\frac{\alpha}{1+\alpha}}K^{\frac{1}{1+\alpha}}\Big)
\,,
\label{eq:pointwise:divcurl:infty}
\end{align}
and, for every $0<\beta<\alpha$,
\begin{align}
\norm{F}_{C^{1,\beta}(B(x_0,R/2))}\le C(R,\alpha,\beta)\,(S+K)
\,.
\label{eq:pointwise:divcurl:holder}
\end{align}
\end{enumerate}
\end{lemma}

\begin{proof}[Proof of Lemma~\ref{lem:pointwise:divcurl}]
For part~(a), set $M=[\varpi]_{C^\alpha(B(R))}$. If one of the points $x,y\in B(R)$ lies on the axis, then the vanishing of $\varpi$ on the axis gives
\[
|\varpi(x)e_\theta(x)-\varpi(y)e_\theta(y)|\le M|x-y|^\alpha.
\]
Suppose now that both points are off the axis. After interchanging them if necessary, assume that $0<r_y\le r_x$, and let $d=|x-y|$. The projections onto the axis belong to $B(R)$, so $|\varpi(y)|\le Mr_y^\alpha$. Moreover,
\[
|e_\theta(x)-e_\theta(y)|\le \min\Bigl\{2,\frac{2d}{r_y}\Bigr\}.
\]
It follows that
\begin{align*}
|\varpi(x)e_\theta(x)-\varpi(y)e_\theta(y)|
&\le |\varpi(x)-\varpi(y)|+|\varpi(y)|\,|e_\theta(x)-e_\theta(y)|\\
&\le M d^\alpha+2M r_y^\alpha\min\{1,d/r_y\}
\le 3M d^\alpha,
\end{align*}
where the last inequality follows separately from $d\ge r_y$ and $d<r_y$. This proves~\eqref{eq:pointwise:vector:holder}.

For part~(b), we write $G=\curl F$. Since $\div F=0$, we have $-\Delta F=\curl G$, so that for every $x_1\in B(x_0,R)$ each component of $F$ solves the Poisson equation $-\Delta F_j=\p_k\big(\varepsilon_{jk\ell}(G_\ell-G_\ell(x_1))\big)$. On a ball $B(x_1,2h)\subset B(x_0,R)$ the interior Schauder estimate~\cite[Theorem~8.32]{GilbargTrudinger01}, applied after rescaling to the unit ball, gives for every $0<\beta<\alpha$
\begin{align}
\norm{\nabla F}_{L^\infty(B(x_1,h))}+h^\beta[\nabla F]_{C^\beta(B(x_1,h))}
\le C(\alpha,\beta)\Big( h^{-1}  \norm{F}_{L^\infty(B(x_1,2h))} + h^\alpha K\Big)
\,.
\label{eq:pointwise:divcurl:scale}
\end{align}
Here we used that $|G-G(x_1)|\le(2h)^\alpha K$ and that $[G]_{C^\beta(B(x_1,2h))}\le(4h)^{\alpha-\beta}K$ on $B(x_1,2h)$. For~\eqref{eq:pointwise:divcurl:infty} we take $\beta=\alpha/2$, $x_1\in B(x_0,R/2)$, and $h=\min\{R/4,(S/K)^{1/(1+\alpha)}\}$, with $h=R/4$ when $K=0$. The right side of~\eqref{eq:pointwise:divcurl:scale} is then at most $C(\alpha,\beta)(S/h+h^\alpha K)$, which is bounded by the right side of~\eqref{eq:pointwise:divcurl:infty} after enlarging $C(\alpha)$. For~\eqref{eq:pointwise:divcurl:holder} we take $h=R/4$ in~\eqref{eq:pointwise:divcurl:scale} at every $x_1\in B(x_0,R/2)$; pairs of points at distance at least $R/4$ are handled by the bound on $\norm{\nabla F}_{L^\infty(B(x_0,R/2))}$.
\end{proof}

\subsection{Proof of Theorem~\ref{thm:sharp}}
\label{sec:sharp:proof}

\begin{proof}[Proof of Theorem~\ref{thm:sharp}]
The last assertion of the theorem is immediate: when $\alpha=1$, since $\omega_\theta$ vanishes on the axis, the second bound in~\eqref{eq:sharp:matched} gives $|\omega_\theta(x,t)|\le C_\alpha\,r/((-t)\ell(t))$, which is the third bound with $C_H=C_\alpha$.

By Proposition~\ref{prop:pointwise:meridional:small}, it suffices to prove that~\eqref{eq:pointwise:expanding:conclusion} holds under the assumption~\eqref{eq:sharp:matched}. The proof of Proposition~\ref{prop:pointwise:expanding} uses the hypotheses~\eqref{eq:pointwise:velocity} only through the bounds~\eqref{eq:pointwise:meridional:bounded} and~\eqref{eq:pointwise:quotient}, through the rescaled bounds~\eqref{eq:pointwise:zoom:bounds}, and through Lemmas~\ref{lem:ancient:axis} and~\ref{lem:ancient:receding}. We first show that~\eqref{eq:pointwise:meridional:bounded} and~\eqref{eq:pointwise:quotient} hold under the assumption~\eqref{eq:sharp:matched}. Then, we record what replaces the Hessian bound~\eqref{eq:pointwise:zoom:bounds:b} after the zoom, and we verify that the two lemmas survive the replacement.

\emph{The gradient and potential vorticity bounds.}
We fix $t\in(-T_1,0)$, with $T_1$ as in~\eqref{eq:pointwise:T:one}, and we apply Lemma~\ref{lem:pointwise:divcurl}(a) on $B(1)$: the second bound in~\eqref{eq:sharp:matched} gives
\begin{align}
[\curl b(\cdot,t)]_{C^\alpha(B(1))}
=[\omega_\theta(\cdot,t)\,e_\theta]_{C^\alpha(B(1))}
\le\frac{3C_\alpha}{(-t)\ell(t)^\alpha}
\,.
\label{eq:sharp:curl:holder}
\end{align}
For $x\in B(3/4)$ we apply~\eqref{eq:pointwise:divcurl:infty} to $b(\cdot,t)$ on $B(x,1/4)\subset B(1)$, with $S\le C_0\ell(t)/(-t)$ by the first bound in~\eqref{eq:sharp:matched}, and with $K\le3C_\alpha/((-t)\ell(t)^\alpha)$ by~\eqref{eq:sharp:curl:holder}. The length $\ell(t)$ cancels in the product $S^{\alpha/(1+\alpha)}K^{1/(1+\alpha)}$, while $4\ell(t)\le1$ by~\eqref{eq:pointwise:T:one}. Taking the supremum over $x\in B(3/4)$ therefore gives~\eqref{eq:pointwise:meridional:bounded} with constant
\[
C_1:=C(\alpha)\Big(C_0+C_0^{\frac{\alpha}{1+\alpha}}(3C_\alpha)^{\frac{1}{1+\alpha}}\Big)
\,.
\]
The bound~\eqref{eq:pointwise:quotient} is the third bound in~\eqref{eq:sharp:matched}, with $C_H$ in place of $\sqrt2C_2$.

\emph{The rescaled bounds.}
The fields and the domains of Section~\ref{sec:pointwise:zoom} are defined as before. The bounds~\eqref{eq:pointwise:zoom:bounds:a},~\eqref{eq:pointwise:zoom:bounds:c}, and~\eqref{eq:pointwise:zoom:bounds:e} hold as stated, and~\eqref{eq:pointwise:zoom:bounds:d} holds with $C_H$ in place of $\sqrt2C_2$. In place of the Hessian bound~\eqref{eq:pointwise:zoom:bounds:b}, the change of variables~\eqref{eq:pointwise:fixed:zoom} applied to the second bound in~\eqref{eq:sharp:matched} and to~\eqref{eq:sharp:curl:holder} gives, for every $\tau\le-1$ and all $n$ large depending on $\tau$,
\begin{align}
[\Theta_n(\cdot,\tau)]_{C^\alpha(\Dcal_n)}\le\frac{C_\alpha}{|\tau|L_n(\tau)^\alpha}\le\frac{C_\alpha}{|\tau|}
\,,\qquad
[\curl B_n(\cdot,\tau)]_{C^\alpha(\Dcal_n)}\le\frac{3C_\alpha}{|\tau|L_n(\tau)^\alpha}\le\frac{3C_\alpha}{|\tau|}
\,.
\label{eq:sharp:zoom:holder}
\end{align}
Here we used that $L_n\ge1$ for $\tau\le-1$.

\emph{The two lemmas.}
We claim that the conclusions of Lemmas~\ref{lem:ancient:axis} and~\ref{lem:ancient:receding} remain valid when the Hessian bound~\eqref{eq:pointwise:zoom:bounds:b} is replaced by~\eqref{eq:sharp:zoom:holder}, with the other rescaled bounds retained as described above. Indeed, the Hessian bound enters the proofs of the two lemmas in \emph{Step~1} and \emph{Step~5} only: \emph{Steps~2, 3, and~4} use, besides the hypotheses of the two lemmas, only the bounds~\eqref{eq:pointwise:zoom:bounds:a},~\eqref{eq:pointwise:zoom:bounds:c},~\eqref{eq:pointwise:zoom:bounds:d}, and~\eqref{eq:pointwise:zoom:bounds:e}, the vanishing of $\eps_n$, and the conclusions of \emph{Step~1}.

In \emph{Step~1} it makes $\curl B_n$ and $\Theta_n$ Lipschitz in $y$, uniformly in $n$ on compact sets. This information is used twice: in the compactness inequality, through the compact embedding $C^{0,1}(\overline{B(\bar y,\rho)})\hookrightarrow C^0(\overline{B(\bar y,\rho)})$, and for the equicontinuity in $y$ of the scalars $\Theta_n$. Under~\eqref{eq:sharp:zoom:holder} the two fields are instead H\"older continuous in $y$ with exponent $\alpha$, uniformly in $n$ on compact sets. Since the embedding $C^{0,\alpha}(\overline{B(\bar y,\rho)})\hookrightarrow C^0(\overline{B(\bar y,\rho)})$ is compact as well, the compactness inequality holds with the $C^{0,\alpha}$ norm in place of the $C^{0,1}$ norm, and the equicontinuity in $y$ follows from the H\"older bound. The rest of \emph{Step~1} is unchanged. In particular, $|\Theta|\le\sqrt2C_1/|\tau|$, and in Lemma~\ref{lem:ancient:axis} the limit $\Omega_\infty$ inherits the bound $C_H/|\tau|$ from~\eqref{eq:pointwise:zoom:bounds:d}, so that the bound $\sqrt2C_2/|\tau|$ used in \emph{Step~4} becomes $C_H/|\tau|$.

In \emph{Step~5}, the bounds~\eqref{eq:pointwise:zoom:bounds:a} and~\eqref{eq:sharp:zoom:holder} at $\tau=-1$, where $L_n(-1)=1$, together with~\eqref{eq:pointwise:divcurl:holder} on unit balls, which lie in $\Dcal_n$ for $n$ large, bound $B_n(\cdot,-1)$ in $C^{1,\beta}$ on every compact subset of $\RR^3$, for every $\beta<\alpha$. This replaces the $C^{1,1}$ bound. A further subsequence again converges in $C^1$ on compact sets, and the rest of \emph{Step~5} is unchanged. This proves the claim.

\emph{Conclusion.}
The proof in Section~\ref{sec:pointwise:applications} then gives~\eqref{eq:pointwise:expanding:conclusion} under the weaker assumption~\eqref{eq:sharp:matched}, and thus Proposition~\ref{prop:pointwise:meridional:small} proves Theorem~\ref{thm:sharp}.
\end{proof}

\section{Sublinear Euler profiles under weaker regularity}
\label{sec:sharp:civ}
In this section we carry out, for the self-similar Euler profiles of~\cite[Section~4]{CIV26}, the weakening of regularity hypotheses announced in Remark~\ref{rem:pointwise:comments}(d). In Section~4 of our previous paper~\cite{CIV26}, we considered axisymmetric, globally self-similar solutions of the 3D Euler equations, with self-similar profile $U$ and similarity exponent $\gamma>0$. Constantin, Ignatova, and Vicol~\cite[Theorem~4.5]{CIV26} proved that $\gamma\ge\frac12$ for every nonzero $C^2$ profile satisfying $U(0)=0$, sublinearity at infinity, and the far-field bounds
\[
|U(y)|\leq C|y|\langle y\rangle^{-1/\gamma},
\qquad
|\nabla U(y)|+|\curl U(y)|\leq C\langle y\rangle^{-1/\gamma}.
\]
That paper also fixes the spatial units by the normalization $\|\nabla^2U\|_{L^\infty}=1$. The normalization does not enter the argument below. Here we retain $U(0)=0$, continuity, sublinearity, local H\"older regularity of the azimuthal vorticity, and local boundedness of the potential vorticity, but impose no far-field decay beyond sublinearity. We use the pressure-free weak form of the self-similar Euler equation $(1-\gamma)U+\gamma y\cdot\nabla U+U\cdot\nabla U+\nabla P=0$. The identity $y\cdot\nabla U=\div(y\otimes U)-3U$ gives~\eqref{eq:CIV:weak} below.

\begin{proposition}[{\bf Sublinear self-similar Euler profiles}]
\label{prop:civ}
Let $\gamma>0$ and let $U\in C(\RR^3;\RR^3)$ be axisymmetric, with $\div U=0$ in $\Dcal'(\RR^3)$, $U(0)=0$, $|U(y)|/|y|\to0$ as $|y|\to\infty$, and
\begin{align}
\int_{\RR^3}\Big((1-4\gamma)\,U\cdot\Xi-\gamma\,(y\otimes U):\nabla\Xi-(U\otimes U):\nabla\Xi\Big)\,dy=0
\,,
\label{eq:CIV:weak}
\end{align}
for all $\Xi\in C^\infty_0(\RR^3;\RR^3)$ with $\div\Xi=0$.
Let $U=b_U+U_\theta e_\theta$, with $b_U=U_re_r+U_ze_z$, be the decomposition of $U$ in the cylindrical frame of the variable $y$, where $r=|y'|$. Assume that $\curl b_U=\omega_\theta e_\theta$ in $\Dcal'(\RR^3)$, where $\omega_\theta$ is an axisymmetric scalar which vanishes on the axis and belongs to $C^{0,\alpha}_{\mathrm{loc}}(\RR^3)$ for some $\alpha\in(0,1]$. Also assume that the potential vorticity $\omega_\theta/r$ is bounded on $B(\rho)\cap\{r>0\}$ for every $\rho>0$; this is automatic when $\alpha=1$. If $U\not\equiv0$, then $\gamma\ge\frac12$. If $U_\theta\equiv0$, then $U\equiv0$ for every $\gamma>0$.
\end{proposition}

\begin{proof}[Proof of Proposition~\ref{prop:civ}]
The proof is carried out in six steps. 

\emph{Step 1: regularity of the meridional part.} Since a continuous axisymmetric field is parallel to $e_z$ on the axis, $U_r$ and $U_\theta$ vanish there, and $b_U$ is continuous on $\RR^3$. The identity $\div(U_\theta e_\theta)=0$ holds in $\Dcal'(\RR^3)$: away from the axis it follows from the cylindrical divergence formula, while integration over the angular variable and the identity $(U_\theta e_\theta)\cdot e_r=0$ exclude an axis contribution. Hence $\div b_U=0$ in $\Dcal'(\RR^3)$.

We fix $\rho>0$, and for $0<\eps\le\rho/2$ set $b_\eps=b_U*\eta_\eps$, where $\eta$ is a radial mollifier supported in $B(1)$. Radial convolution preserves rotation equivariance and reflection symmetry across planes containing the axis. Thus $b_\eps$ is smooth, axisymmetric, and meridional, with $\div b_\eps=0$ and $\curl b_\eps=(\omega_\theta e_\theta)*\eta_\eps$. By~\eqref{eq:pointwise:vector:holder},
\[
[\curl b_\eps]_{C^\alpha(B(\rho/2))}\le3[\omega_\theta]_{C^\alpha(B(\rho))},
\qquad
\norm{b_\eps}_{L^\infty(B(\rho/2))}\le\norm{b_U}_{L^\infty(B(\rho))}.
\]
The bound~\eqref{eq:pointwise:divcurl:holder} on $B(\rho/2)$ therefore bounds $b_\eps$ in $C^{1,\beta}(B(\rho/4))$ uniformly in $\eps$, for every $\beta<\alpha$. For a fixed $\beta<\alpha$, we apply this statement with an exponent strictly between $\beta$ and $\alpha$. Since $b_\eps\to b_U$ uniformly on compact sets, compactness gives $b_U\in C^{1,\beta}_{\mathrm{loc}}(\RR^3)$ and $b_\eps\to b_U$ in $C^1$ on compact sets. In particular, $V:=\gamma y+b_U$ is locally Lipschitz, with $\div V=3\gamma$.

\emph{Step 2: characteristics.} Since $V$ is locally Lipschitz, the equation $\dot Y=V(Y)$ has unique solutions. By \emph{Step~1}, $U_r=0$ on the axis, so that $V$ is parallel to $e_z$ there, the axis is invariant, and, by uniqueness, a solution which starts off the axis stays off it. By sublinearity we may fix $R_\flat$ with $|U(y)|\le\frac\gamma2|y|$ for $|y|\ge R_\flat$; since $|b_U|\le|U|$, we then have $\frac{d}{d\tau}|Y|^2=2\gamma|Y|^2+2Y\cdot b_U(Y)\ge\gamma|Y|^2$ whenever $|Y|\ge R_\flat$, so that $|Y(\tau)|\le\max\{|Y(0)|,R_\flat\}$ for $\tau\le0$. The standard continuation criterion for locally Lipschitz ordinary differential equations now shows that the backward characteristic starting at a point $y_0$ off the axis exists on $(-\infty,0]$, stays in the compact set $K_0:=\overline{B(\max\{|y_0|,R_\flat\})}$, and stays off the axis.

\emph{Step 3: a stationary form of Lemma~\ref{lem:pointwise:characteristics}.} Let $\Ecal\subset\RR^3$ be open, $V\colon\Ecal\to\RR^3$ locally Lipschitz with $\div V=c$, $f\in C(\Ecal)$, $\kappa\in\RR$, and
\begin{align}
\int_{\Ecal}f\,\big(\kappa\psi-V\cdot\nabla\psi\big)\,dy=0\qquad\mbox{for all }\psi\in C^\infty_0(\Ecal)
\,.
\label{eq:CIV:stationary:weak}
\end{align}
Equation~\eqref{eq:CIV:stationary:weak} is the weak form of  $V\cdot\nabla f+\mu f=0$,  for $\mu:=\kappa+c$. If $Y\colon[\tau_0,\tau_1]\to\Ecal$ solves $\dot Y=V(Y)$, then $e^{\mu\tau}f(Y(\tau))$ is constant. In order to prove this fact, we let $N\subset\Ecal$ be a compact neighborhood of $Y([\tau_0,\tau_1])$, we take $\delta<\mathrm{dist}(N,\Ecal^\complement)$, and we set $f^\delta=f*\eta_\delta$ on $N$. Testing~\eqref{eq:CIV:stationary:weak} with $\psi=\eta_\delta(y-\cdot)$ and using $\int(V(y)-V(\zeta))\cdot(\nabla\eta_\delta)(y-\zeta)\,d\zeta=-c$, due to   $\div V=c$, we obtain that on $N$:
\[
V\cdot\nabla f^\delta+\mu f^\delta
=\int\big(f(\zeta)-f(y)\big)\big(V(y)-V(\zeta)\big)\cdot(\nabla\eta_\delta)(y-\zeta)\,d\zeta+c\,(f^\delta-f)=:\mathsf{e}_\delta
\,.
\]
The error term $\mathsf{e}_\delta$ may be directly bounded as $|\mathsf{e}_\delta|\leq C\,\nu(\delta)\,(L_N+|c|)$, where $\nu$ is the modulus of continuity of $f$, and $L_N$ is the Lipschitz constant of $V$, both on the closed $\delta$-neighborhood of $N$, which is a compact subset of $\Ecal$. Up to sign, $\mathsf{e}_\delta$ is the commutator $(V\cdot\nabla f)*\eta_\delta-V\cdot\nabla f^\delta$ of DiPerna and Lions~\cite[Lemma~II.1]{DiPernaLions89}. Since $\frac{d}{d\tau}\big(e^{\mu\tau}f^\delta(Y(\tau))\big)=e^{\mu\tau}\mathsf{e}_\delta(Y(\tau))$, letting $\delta\to0$ proves the claim.

We shall also use this statement when the scalar identity is initially known only for axisymmetric test functions. Suppose that $\Ecal$ is invariant under rotations $R_\vartheta$ about the $z$-axis, $f$ is axisymmetric, and $V(R_\vartheta y)=R_\vartheta V(y)$. For an arbitrary $\psi\in C^\infty_0(\Ecal)$, define
\[
\overline\psi(y)=\frac1{2\pi}\int_0^{2\pi}\psi(R_\vartheta y)\,d\vartheta.
\]
A change of variables shows that the left side of~\eqref{eq:CIV:stationary:weak} has the same value at $\psi$ and at $\overline\psi$. Therefore, validity for all axisymmetric test functions implies validity for all test functions.

\emph{Step 4: the swirl.} We set $\Gamma:=rU_\theta=y_1U_2-y_2U_1$, which is continuous on $\RR^3$. For an axisymmetric $\psi\in C^\infty_0(\RR^3)$, the test function $\Xi=\psi\,(e_z\times y)=\psi\,r e_\theta$ is divergence free, and satisfies $\p_i\Xi_j=\p_i\psi\,(e_z\times y)_j+\psi\,\varepsilon_{j3i}$. Moreover, $U\cdot\nabla\psi=b_U\cdot\nabla\psi$ because $\psi$ is axisymmetric. With this test function $\Xi$, equation~\eqref{eq:CIV:weak} gives
\[
(1-5\gamma)\int\Gamma\psi-\int\Gamma\,V\cdot\nabla\psi=0.
\]
By the rotational averaging observation in \emph{Step~3}, this identity holds for every $\psi\in C^\infty_0(\RR^3)$. Step~3 with $\Ecal=\RR^3$, $\kappa=1-5\gamma$, $c=3\gamma$, and $\mu=1-2\gamma$ gives $e^{(1-2\gamma)\tau}\Gamma(Y(\tau))=\Gamma(y_0)$ for $\tau\le0$. Since $\Gamma$ is bounded on $K_0$, if $\gamma<\frac12$, then letting $\tau\to-\infty$ gives $\Gamma(y_0)=0$ for every $y_0$ off the axis, so that $U_\theta\equiv0$ and $U=b_U$.

\emph{Step 5: the potential vorticity.} Assume $U=b_U$, and let $\chi\in C^\infty_0(\{r>0\})$ be axisymmetric. Set $\Xi=\curl(\chi e_\theta)$, $\psi:=r\chi$, and $q:=\omega_\theta/r$, which is continuous on $\{r>0\}$. We identify the three terms of~\eqref{eq:CIV:weak}. First,
\[
\int b_U\cdot\Xi=\int\omega_\theta\chi=\int q\psi.
\]
To justify the next two identities at the stated regularity, we write $\curl b_\eps=\omega_{\theta,\eps}e_\theta$. By \emph{Step~1}, $b_\eps\to b_U$ in $C^1$ on the compact support of $\chi$, and hence $\omega_{\theta,\eps}\to\omega_\theta$ uniformly there. Using $\curl(y\cdot\nabla b_\eps)=\curl b_\eps+y\cdot\nabla\curl b_\eps$ and $y\cdot\nabla e_\theta=0$, we obtain
\begin{align*}
-\int(y\otimes b_U):\nabla\Xi
&=\lim_{\eps\to0}\int(3b_\eps+y\cdot\nabla b_\eps)\cdot\Xi\\
&=\lim_{\eps\to0}\int(4\omega_{\theta,\eps}+y\cdot\nabla\omega_{\theta,\eps})\chi\\
&=\int\omega_\theta(\chi-y\cdot\nabla\chi)
=\int q\,(2\psi-y\cdot\nabla\psi).
\end{align*}
Similarly, since $\div b_\eps=0$, $b_\eps\cdot\nabla b_\eps=\curl b_\eps\times b_\eps+\nabla(|b_\eps|^2/2)$, and $\div\Xi=0$, we have
\begin{align*}
-\int(b_U\otimes b_U):\nabla\Xi
&=\lim_{\eps\to0}\int(b_\eps\cdot\nabla b_\eps)\cdot\Xi\\
&=\lim_{\eps\to0}\int(\omega_{\theta,\eps}e_\theta\times b_\eps)\cdot\curl(\chi e_\theta)\\
&=\int(\omega_\theta e_\theta\times b_U)\cdot\curl(\chi e_\theta)
=-\int q\,b_U\cdot\nabla\psi,
\end{align*}
where the last equality uses the pointwise identity
\[
(e_\theta\times b_U)\cdot\curl(\chi e_\theta)
=-r^{-1}b_U\cdot\nabla(r\chi).
\]
Summing the three terms gives
\[
(1-2\gamma)\int q\psi-\int q\,V\cdot\nabla\psi=0
\]
for every axisymmetric $\psi\in C^\infty_0(\{r>0\})$, since any such $\psi$ has the form $\psi=r\chi$ with $\chi=\psi/r\in C^\infty_0(\{r>0\})$. By the rotational averaging observation in \emph{Step~3}, the identity holds for every test function in $C^\infty_0(\{r>0\})$. Step~3 with $\Ecal=\{r>0\}$, $\kappa=1-2\gamma$, $c=3\gamma$, and $\mu=1+\gamma$ gives $e^{(1+\gamma)\tau}q(Y(\tau))=q(y_0)$ for $\tau\le0$. Since $q$ is bounded on $K_0\cap\{r>0\}$ by hypothesis, letting $\tau\to-\infty$ gives $q(y_0)=0$. Thus $\omega_\theta\equiv0$ on $\{r>0\}$, hence on $\RR^3$ by continuity.

\emph{Step 6: conclusion.} We have shown that $\curl U=\curl b_U=0$ and $\div U=0$ in $\Dcal'(\RR^3)$, so $\Delta U=0$ in distributions. Thus, each component of $U$ is smooth and harmonic. A harmonic function with sublinear growth is constant, and $U(0)=0$ gives $U\equiv0$. The assumption $\gamma<\frac12$ was used in \emph{Step~4} only; therefore, this proof also shows that a swirl-free profile vanishes for every $\gamma>0$.
\end{proof}

\section*{Acknowledgments}
The work of P.C. was partially supported by NSF grant DMS-2606072. The work of M.I. was partially supported by NSF grant DMS-2204614. The work of V.V. was partially supported by the Collaborative NSF grant DMS-2307681 and by a Simons Investigator Award.

\section*{Disclaimer on the use of AI systems} 
Publicly available versions of Claude (Opus 5 and Fable 5.1, within Claude Code) and ChatGPT (Sol and Astra, within Codex) were used in the preparation of this manuscript. These systems were given access to our paper~\cite{CIV26} (which was written without the use of AI systems), from which the core idea of the proof is taken, and to an early draft of this manuscript, written entirely by the authors, which contained the whole proof at the level of an eight-page sketch. Claude and ChatGPT were used to fill in the proofs of routine lemmas and arguments in that draft, to perform an extended bibliographic search, and to produce the TikZ figure. Over the course of two extended sessions of Claude Code and Codex, the manuscript was then polished, in its exposition and in the verification of its mathematical computations, under the close guidance and supervision of the authors. Every suggestion, computation, and reference produced by these systems was independently verified and, where necessary, corrected or reformulated by the authors, who take full responsibility for the manuscript as it is presently written.

\end{document}